\documentclass[reqno,oneside,10pt]{amsart}
\usepackage{lmodern}
\usepackage{microtype}
\usepackage[latin9]{inputenc}
\usepackage[normalem]{ulem}
\usepackage{amsmath,amsfonts,amssymb,amsfonts,amscd,amsxtra}
\usepackage{enumerate}
\usepackage[shortlabels]{enumitem}
\usepackage{mathrsfs}
\usepackage{mathtools}
\usepackage{fancyhdr}
\usepackage{xspace}
\usepackage{geometry}
\usepackage{tikz}
\tikzset{node distance=1.5cm, auto}
\usepackage{color}
\allowdisplaybreaks
\usepackage{latexsym}
\usepackage{graphicx}
\usepackage[all]{xy}
\usepackage{version}

\makeatletter
\DeclareMathSizes{\@xpt}{\@xpt}{6}{5}
\makeatother

\newcommand{\Hom}{{\sf Hom}}

\newcommand{\End}{{\sf End}}
\newcommand{\Id}{{\sf Id}}
\newcommand{\GL}{{\sf GL}}
\def\CC{{\mathbb C}}
\def\NN{{\mathbb N}}
\newcommand{\Pp}{\mathcal{P}}

\newcommand{\Sym}{\mathfrak{S}}
\newcommand{\Ind}{\mathsf{Ind}}
\newcommand{\Res}{\mathsf{Res}}
\newcommand{\PInd}{\mathsf{PInd}}
\newcommand{\PPInd}{\overline{\mathsf{PInd}}}

\newcommand{\ie}{i.e.~}
\newcommand{\eg}{e.g.~}

\newcommand{\tensor}[1]{\otimes_{#1}}
\newcommand{\GPRep}[2]{\mathsf{PRep}_{#1}^{#2}}
\newcommand{\parrep}{\mathsf{PRep}}

\excludeversion{invisible*}

\newcommand{\C}{{\mathbb C}}

\newtheorem{prop}{Proposition}[section]
\newtheorem{proposition}[prop]{Proposition}
\newtheorem{lemma}[prop]{Lemma} 
 
\newtheorem{corollary}[prop]{Corollary}

\newtheorem{theorem}[prop]{Theorem}

\theoremstyle{definition}

\newtheorem{definition}[prop]{Definition}

\newtheorem{example}[prop]{Example}

\newtheorem{remark}[prop]{Remark}
\newtheorem{remarks}[prop]{Remarks}

\newtheorem{notation}[prop]{Notation}

\numberwithin{equation}{section}

\makeatletter
\newcommand\notni{\mathrel{\m@th\mathpalette\canc@l\owns}}
\newcommand\canc@l[2]{{\ooalign{$\hfil#1/\mkern1mu\hfil$\crcr$#1#2$}}}
\makeatother

\makeatother

\begin{document}

\title{Partial and global representations of finite groups}

\author{Michele D'Adderio}
\address{Universit\'e Libre de Bruxelles (ULB)\\D\'epartement de Math\'ematique\\ Boulevard du Triomphe, B-1050 Bruxelles\\ Belgium}\email{mdadderi@ulb.ac.be}

\author{William Hautekiet}
\address{Universit\'e Libre de Bruxelles (ULB)\\D\'epartement de Math\'ematique\\ Boulevard du Triomphe, B-1050 Bruxelles\\ Belgium}\email{william.hautekiet@ulb.ac.be}

\author{Paolo Saracco}
\address{Universit\'e Libre de Bruxelles (ULB)\\D\'epartement de Math\'ematique\\ Boulevard du Triomphe, B-1050 Bruxelles\\ Belgium}\email{paolo.saracco@ulb.ac.be}

\author{Joost Vercruysse}
\address{Universit\'e Libre de Bruxelles (ULB)\\D\'epartement de Math\'ematique\\ Boulevard du Triomphe, B-1050 Bruxelles\\ Belgium}\email{jvercruy@ulb.ac.be}

\begin{abstract} 
Given a subgroup $H$ of a finite group $G$, we begin a systematic study of the partial representations of $G$ that restrict to global representations of $H$. After adapting several results from \cite{Dokuchaev-Exel-Piccione} (which correspond to the case $H=\{1_G\}$), we develop further an effective theory that allows explicit computations. As a case study, we apply our theory to the symmetric group $\mathfrak{S}_n$ and its subgroup $\mathfrak{S}_{n-1}$ of permutations fixing $1$: this provides a natural extension of the classical representation theory of $\mathfrak{S}_n$. 
\end{abstract}

\maketitle

\tableofcontents

\section*{Introduction}

The notion of a \emph{partial action} of a group arose first in the theory of operator algebras as an approach to $C^*$-algebras generated by partial isometries, allowing the study of their $K$-theory, ideal structure and representations \cite{Exel-CircleActions-1994}. In particular, the point of view of crossed products by partial actions of groups was very successful for classifying $C^*$-algebras.
The related notion of \emph{partial representation} has been introduced in \cite{Dokuchaev-Exel-TAMS-2005}. Since then, these notions have been studied and applied in a variety of contexts involving operator algebras, dynamical systems, commutative algebras, noncommutative rings and Hopf algebras, among others. See \cite{Batista-Survey-2017} and \cite{Dokuchaev-Survey-2019} for two recent surveys. 

\medskip

Our investigation stemmed from the article \cite{Dokuchaev-Exel-Piccione}, in which the authors give a first systematic approach to the theory of partial representations of finite groups. Among the main results in \cite{Dokuchaev-Exel-Piccione}, there is the proof of an equivalence between the category of partial representations of a finite group $G$ over a field (say $\C$, for instance, although the actual statement is more general) and the category of (usual) representations of the so called \emph{partial group algebra} $\C_{par}G$, which is a unital associative algebra. In \cite{Dokuchaev-Exel-Piccione} it is also shown that $\C_{par}G$ is isomorphic to the groupoid algebra $\C\Gamma(G)$ of a certain groupoid $\Gamma(G)$ associated to $G$. The algebra $\C_{par}G\cong \C\Gamma(G)$ is proved to be semisimple and a formula is provided in \cite{Dokuchaev-Exel-Piccione}, which describes it as a direct product of matrix algebras of the form $M_m(\C K)$ for $K$ varying among the subgroups of $G$.

It turns out that an explicit computation of such a formula for a given finite group $G$ (or even, more importantly, a description of its irreducible partial representations) seems to be in general out of reach, the problem being mainly that the number of summands $M_m(\C K)$ grows rapidly with $|G|$. Already in \cite{Dokuchaev-Exel-Piccione} it is shown that even for abelian groups, whose global irreducible representations are fairly easy to describe, the computation of their irreducible partial representations quickly becomes way too involved.

\medskip

Unsatisfied with this state of affairs, we made the following two related observations, that triggered the present work. First of all, by looking at natural examples of partial representations of a finite group $G$, it is often the case that these partial representations restrict to global (i.e.\ usual) representations of nontrivial subgroups $H$ of $G$. Therefore, to understand for example their decomposition into irreducibles, it would be enough to know the irreducibles among all the partial representations of $G$ that restrict to global representations of $H$, which we call \emph{$H$-global $G$-partial representations}. Observe that for $H$ the trivial subgroup of $G$, we recover exactly the notion of partial representations of $G$.

The second observation is that, if $H$ is relatively large in $G$, then the problem of describing the irreducible $H$-global $G$-partial representations becomes actually tractable and, in our opinion, quite interesting.

\medskip 

In the present article we initiate a systematic study of $H$-global $G$-partial representations of finite groups. Our aim is to build an effective theory that allows explicit computations.

With this goal in mind, we start by adapting most of the results in \cite{Dokuchaev-Exel-Piccione} to our more general situation. In details, we define the analogues of the partial group algebra $\C_{par}G$ and of the groupoid $\Gamma(G)$, that we denote by $\C_{par}^HG$ and $\Gamma_H(G)$ respectively. Then, we prove theorems whose statements are similar to the aforementioned ones, showing for example that the category of $H$-global $G$-partial representations is equivalent to the category of usual representations of the associative unital algebra $\C_{par}^HG\cong \C\Gamma_H(G)$ and that such an algebra is semisimple, by providing a formula that exhibits it as a direct product of algebras of the form $M_m(\C K)$ for certain subgroups $K$ of $G$ that we describe.

After these natural steps, we go deeper into the theory, by achieving the following results:
\begin{itemize}
	\item we give an explicit construction of all the irreducible $H$-global $G$-partial representations in terms of the irreducible representations of the subgroups $K$ appearing in the aforementioned formula;
	
	\item we give a formula for the decomposition into (global) representations of $H$ of the restriction to $H$ of the aforementioned $H$-global $G$-partial irreducibles;
	
	\item we define a notion of \emph{globalization} of partial representations and we prove that every partial representation admits a globalization (giving the analogue of Abadie's theorem \cite{Abadie-globalization-2003} for partial actions);
	
	\item we describe explicitly the globalization of our irreducible $H$-global $G$-partial representations;
	
	\item we prove the existence of an \emph{induced} partial representation of a (global) representation of $H$ to an $H$-global $G$-partial representation, giving a \emph{Frobenius reciprocity};
	
	\item we describe a semigroup $S_H(G)$ that plays the same role for $H$-global $G$-partial representations as the semigroup $S(G)$, defined by Exel in \cite{Exel-semigroup-PAMS}, does for $G$-partial representations. 
\end{itemize}

To make our case for the study of $H$-global $G$-partial representations, we apply all the theory that we developed to the irreducible partial representations of the symmetric group $\mathfrak{S}_n$ which restrict to global representations of the subgroup of the permutations that fix $1$, which we identify with $\mathfrak{S}_{n-1}$. Observe that in order to understand all the irreducible partial representations of $\mathfrak{S}_{n}$, according to the formula in \cite{Dokuchaev-Exel-Piccione}, we essentially need to understand the irreducible representations of all the subgroups of $\mathfrak{S}_n$: by a well-known theorem of Cayley, this boils down to understand the irreducible representations of every finite group, which is obviously a hopeless task. On the other hand, the irreducible $\mathfrak{S}_{n-1}$-global $\mathfrak{S}_{n}$-partial representations can be described explicitly and they provide a natural extension of the classical representation theory of $\mathfrak{S}_{n}$. In this vein, we also prove a \emph{branching rule} in this more general setting.

\medskip

The paper is organized in the following way. 

In Section~\ref{sec:partialactions} we recall some basic definitions. In Section~\ref{sec:globalization} we prove that every partial representation has a globalization. In Section~\ref{sec:HGbasics} we give the basic definitions and examples of $H$-global $G$-partial representations. In Section~\ref{sec:genconstruction} we give a general construction that provides a large class of examples of $H$-global $G$-partial representations. In Section~\ref{sec:GeneralProperties} we prove some basic identities about partial representations that will be useful in the remaining sections. In Section~\ref{sec:ParRepMods} we prove that the category of $H$-global $G$-partial representations is equivalent to the category of left modules over the algebra $\C_{par}^HG$. In Section~\ref{sec:groupoid} we prove that $\C_{par}^HG$ is isomorphic to the groupoid algebra $\C\Gamma_H(G)$. In Section~\ref{sec:irreducibles} we construct all the irreducible $H$-global $G$-partial representations. In Section~\ref{sec:restrictionH} we give a formula for the restriction to $H$ of the aforementioned irreducibles. In Section~\ref{sec:gloablirreps} we describe the globalization of the aforementioned irreducibles. In Section~\ref{sec:induction} we introduce a partial induced representation of a (global) representation of $H$, and we prove a Frobenius reciprocity in this setting. In Section~\ref{sec:semigroup} we introduce a semigroup $S_H(G)$ that plays the role of the semigroup $S(G)$ of Exel for the $H$-global $G$-partial representations. In Section~\ref{sec:SnSnminus1} we apply all our theory to the interesting case $G=\Sym_n$ and $H=\Sym_{n-1}$, providing a natural extension of the classical representation theory of $\Sym_n$. In Section~\ref{sec:branching} we prove a branching rule in this more general setting. Finally in Section~\ref{sec:comments} we give some further comment and we indicate some possible future directions.

\section*{Acknowledgments}

This paper was written while P. Saracco was member of the ``National Group for Algebraic and Geometric Structures and their Applications'' (GNSAGA-INdAM). He acknowledges FNRS support through a postdoctoral fellowship as Collaborateur Scientifique in the framework of the project ``(CO)REPRESENTATIONS'' (application number 34777346). \\
This research was partially funded by FNRS under the MIS project ``Antipode''  (Grant F.4502.18).

\medskip

\emph{Though many definitions and results work in a more general setting, in this article all groups are finite and all vector spaces are over $\C$ and finite-dimensional, unless otherwise stated.}

\section{Partial actions and partial representations}

In this section we discuss some general results for partial actions and partial representations.

\subsection{Generalities} \label{sec:partialactions}

We begin by recalling some basic notions that are essential in the rest of this article. We refer the reader to \cite[Chapters 2 and 3]{Exel-book-2017} for further details on these topics.

\medskip

The following definition is due to Exel.
\begin{definition}[{\cite[Definition 1.2]{Exel-semigroup-PAMS}}]
	A \emph{partial action} $\alpha=(\{X_g\}_{g\in G},\{\alpha_g\}_{g\in G})$ of a group $G$ (also called \emph{$G$-partial action}) on a set $X$ consists of a family of subsets $X_g\subseteq X$ indexed by $G$ and a family of bijections $\alpha_g\colon X_{g^{-1}}\to X_g$ for each $g\in G$, satisfying the following conditions:
	\begin{enumerate}[(i)]
		\item $X_{1_G}=X$ and $\alpha_{1_G}=\Id_X$;
		\item $\alpha_h^{-1}(X_{g^{-1}}\cap X_h)\subseteq X_{(gh)^{-1}}$ for every $g,h\in G$;
		\item $\alpha_g(\alpha_h(x))=\alpha_{gh}(x)$ for every $x\in \alpha_h^{-1}(X_{g^{-1}}\cap X_h)$.
	\end{enumerate}
	A \emph{morphism} between two partial actions $\alpha=(\{X_g\}_{g\in G},\{\alpha_g\}_{g\in G})$ and  $\beta=(\{Y_g\}_{g\in G},\{\beta_g\}_{g\in G})$ of the same group $G$ is a function $\phi\colon X\to Y$ such that 
	\begin{enumerate}[(a)]
		\item $\phi(X_g)\subseteq Y_g$ for every $g\in G$, and
		\item for every $x\in X_{g^{-1}}$, $\beta_g(\phi(x))=\phi(\alpha_g(x))$.
	\end{enumerate}
\end{definition}
\begin{remark}
	It follows easily from the previous definition that $\alpha_g^{-1}=\alpha_{g^{-1}}$ and that, in fact,	for all $g,h\in G$ \[\alpha_g(X_{g^{-1}}\cap X_h)=X_g\cap X_{gh}. \]
\end{remark}

Clearly a global action of a group $G$ on a set $X$ is in particular a partial action with $X_g=X$ for all $g\in G$. A typical example of a partial action is obtained by restricting a global action to a proper subset.
\begin{definition}[{\cite[Example 1.1]{Abadie-globalization-2003}}]
	Given a global action $(Y,\beta)$ of a group $G$, i.e.\ a homomorphism $\beta\colon G\to \mathsf{Sym}(Y)$ into the symmetric group on $Y$ (i.e.\ the group of the bijections of $Y$ into itself), and a subset $X\subseteq Y$, we can define the \emph{restriction} $\alpha=(\{X_g\}_{g\in G},\{\alpha_g\}_{g\in G})$ of $(Y,\beta)$ to $X$ by setting 
	\begin{enumerate}[label=(R\arabic*),ref=(R\arabic*)]
		\item\label{item:R1} $X_g\coloneqq X\cap \beta_g(X)$ for every $g\in G$, and
		\item\label{item:R2} $\alpha_g\colon X_{g^{-1}}\to X_g$, $\alpha_g(x)\coloneqq \beta_g(x)$ for every $g\in G$ and $x\in X_{g^{-1}}$.
	\end{enumerate} 
	It is easy to check that this is indeed a partial action of $G$ on $X$.
\end{definition}
It turns out that every partial action can be obtained in this way from a suitable global action. To state the precise result, we recall another definition.
\begin{definition}
	A \emph{globalization} (also called \emph{enveloping action} in \cite[Definition 1.2]{Abadie-globalization-2003}) of a partial action $\alpha$ of a group $G$ on a set $X$ is a triple $(Y,\beta,\varphi)$, in which
	\begin{enumerate}[label=(GL\arabic*),ref=(GL\arabic*)]
		\item\label{item:GL1} $Y$ is a set and $\beta\colon G\to \mathsf{Sym}(Y)$ is an action of $G$ on $Y$;
		\item\label{item:GL2} $\varphi\colon X\to Y$ is an injective map;
		\item\label{item:GL3} for every $g\in G$, $\varphi(X_g)=\varphi(X)\cap \beta_g(\varphi(X))$;
		\item\label{item:GL4} for every $x\in X_{g^{-1}}$, we have $\varphi(\alpha_g(x))=\beta_g(\varphi(x))$;
		\item\label{item:GL5} $Y=\bigcup_{g\in G}\beta_g(\varphi(X))$.
	\end{enumerate}
\end{definition}
\begin{remark}
	Sometimes in the literature the triples not satisfying axiom \ref{item:GL5} are already called globalizations, while the ones that satisfy it are designated as \emph{admissible} (cf. \cite[Definition 4.1]{Batista-Survey-2017}). We prefer our linguistic simplification, as this is the only notion that we will use in this work.
\end{remark}
\begin{remark}
	Observe that properties \ref{item:GL1}-\ref{item:GL4} amount to say that $\varphi$ is an isomorphism between $\alpha$ and the restriction of $(Y,\beta)$ to $\varphi(X)\subseteq Y$.
\end{remark}
The following theorem, in the context of continuous partial group actions on topological spaces (and, in particular, on abstract sets), is due to Abadie \cite[Theorem 1.1]{Abadie-globalization-2003}. For a proof in the framework under consideration, we refer to \cite[Theorem 3.5]{Exel-book-2017}.
\begin{theorem} \label{thm:abadie}
	There exists a globalization for every partial action $\alpha$ of a group $G$ on a set $X$ and it is unique up to isomorphism. It can be explicitly realized as $Y = G \times X/ \sim$ where $(g,x) \sim (h,y)$ if and only if $x \in X_{g^{-1}h}$ and $\alpha_{h^{-1}g}(x)=y$.
\end{theorem}
\begin{remark} \label{rem:globuniversal}
	After accepting its existence, the uniqueness of the globalization $(Y,\beta,\varphi)$ of $\alpha$ follows easily from the following universal property: for every triple $(Y',\beta',\varphi')$ satisfying \ref{item:GL1}-\ref{item:GL4}, the function $\psi\colon Y\to Y'$ given by $\psi(\beta_g(\varphi(x)))\coloneqq \beta_g'(\varphi'(x))$ for all $g\in G$ and $x\in X$ is well defined. In particular, it follows immediately from \ref{item:GL5} that $\psi$ is uniquely determined by its defining property, that it satisfies $\psi\circ \varphi=\varphi'$ and that it is a morphism of $G$-sets. 
	
	To see why $\psi$ is well defined, observe that if $\beta_g(\varphi(x))=\beta_h(\varphi(y))$ for some $g,h\in G$ and $x,y\in X$, then $\beta_{h^{-1}}\beta_g(\varphi(x))=\beta_{h^{-1}g}(\varphi(x))=\varphi(y)$, which, by \ref{item:GL4} and \ref{item:GL2}, implies that $\alpha_{h^{-1}g}(x)=y$, so that $\beta_{h^{-1}}'\beta_g'(\varphi'(x))=\beta_{h^{-1}g}'(\varphi'(x))=\varphi'(\alpha_{h^{-1}g}(x))=\varphi'(y)$, which gives $\beta_g'(\varphi'(x))=\beta_h'(\varphi'(y))$ as we wanted.
\end{remark}

\noindent We recall the notion of partial representation of a group on a vector space, by adapting \cite[Definition 6.1]{Dokuchaev-Exel-TAMS-2005}.
\begin{definition} \label{def:partial_rep}
	A \emph{partial representation} $(V,\pi)$ of a group $G$ (also called \emph{$G$-partial representation}) on a vector space $V$ is a map $\pi\colon G\to \End(V)$ such that for all $g, h\in G$
	\begin{enumerate}[label=(PR\arabic*),ref=(PR\arabic*)]
		\item\label{item:PR1} $\pi(1_G)=\Id_V$;
		\item\label{item:PR2} $\pi(g^{-1})\pi(gh)=\pi(g^{-1})\pi(g)\pi(h)$;
		\item\label{item:PR3} $\pi(gh)\pi(h^{-1})=\pi(g)\pi(h)\pi(h^{-1})$.
	\end{enumerate}
	We will informally say that $\pi$ behaves as a group homomorphism ``in the presence of a witness''.
	
	Given two partial representations $(V,\pi)$ and $(V',\pi')$ of the same group $G$, a \emph{morphism} between them is a linear map $f\colon V\to V'$ such that $f\circ \pi(g)=\pi'(g)\circ f$ for all $g\in G$. In particular, partial representations of a group form a category that we denote by $\parrep_G$.
\end{definition}

\begin{lemma} \label{lem:idemp}
	The elements $\pi(g)\pi(g^{-1})$ of $\End(V)$ are commuting idempotents, that is to say,
	\begin{equation} \label{eq:commidem}
	\pi(g)\pi(g^{-1})\pi(g)\pi(g^{-1}) = \pi(g)\pi(g^{-1}) \qquad \text{and} \qquad \pi(g)\pi(g^{-1})\pi(h)\pi(h^{-1}) = \pi(h)\pi(h^{-1})\pi(g)\pi(g^{-1})
	\end{equation}
	for all $g,h \in G$. In particular,
	\begin{equation}\label{eq:projections}
	v \in \pi(g)\pi(g^{-1})(V) \qquad \text{if and only if} \qquad v = \pi(g) \pi(g^{-1}) (v). 
	\end{equation}
	Moreover
	\begin{equation} \label{eq:idemp}
	\pi(g)\pi(h)\pi(h^{-1})=\pi(gh)\pi(h^{-1}g^{-1})\pi(g) .
	\end{equation}
\end{lemma}

\begin{proof}
	From the defining property of partial representations, we get 
	\[ \pi(g)\pi(g^{-1})\pi(g)\pi(g^{-1}) \stackrel{\ref{item:PR2}}{=} \pi(g)\pi(1_G)\pi(g^{-1}) \stackrel{\ref{item:PR1}}{=} \pi(g)\pi(g^{-1}). \]
	Also
	\begin{align*}
	\pi(g)\pi(g^{-1})\pi(h)\pi(h^{-1}) & \stackrel{\ref{item:PR2}}{=} \pi(g)\pi(g^{-1} h)\pi(h^{-1}) = \pi(hh^{-1}g)\pi(g^{-1} h)\pi(h^{-1}) \stackrel{\ref{item:PR3}}{=} \pi(h)\pi(h^{-1}g)\pi(g^{-1} h)\pi(h^{-1}) \\
	& \stackrel{\ref{item:PR2}}{=} \pi(h)\pi(h^{-1}g)\pi(g^{-1}) \stackrel{\ref{item:PR3}}{=} \pi(h)\pi(h^{-1})\pi(g)\pi(g^{-1}),
	\end{align*}
	which proves the first claim, and
	\begin{align*}
	\pi(g)\pi(h)\pi(h^{-1}) & \stackrel{\ref{item:PR2}}{=} \pi(g)\pi(g^{-1})\pi(g)\pi(h)\pi(h^{-1}) \stackrel{\eqref{eq:commidem}}{=} \pi(g)\pi(h)\pi(h^{-1})\pi(g^{-1})\pi(g)\\
	& \stackrel{\ref{item:PR3}}{=} \pi(gh)\pi(h^{-1})\pi(g^{-1})\pi(g) \stackrel{\ref{item:PR3}}{=} \pi(gh)\pi(h^{-1}g^{-1})\pi(g),
	\end{align*}
	which proves the third one. Concerning \eqref{eq:projections}, notice that if $v \in \pi(g)\pi(g^{-1})(V)$ then $v = \pi(g)\pi(g^{-1})(w)$ for some $w \in V$ and hence
	\[
	\pi(g) \pi(g^{-1}) (v) = \pi(g)\pi(g^{-1})\pi(g)\pi(g^{-1})(w) \stackrel{\eqref{eq:commidem}}{=} \pi(g)\pi(g^{-1})(w) = v.
	\]
	The other implication is obvious.
\end{proof}

It is clear that global (i.e.\ usual) representations are partial representations. Moreover a partial representation $(V,\pi)$ of a group $G$ is a global representation if and only if $\pi(g)$ is invertible for all $g\in G$. 

A natural example of a partial representation is given by the linearization of a partial action.
\begin{definition}\label{s1def:linearization}
	Given a set $X$, let $\C[X]$ be the vector space over $\C$ with basis $X$. For any subset $Y\subseteq X$, let $P_Y\colon \C[X]\to \C[Y]$ be the obvious projection whose kernel is $\C[X\setminus Y]$. Given a partial action $\alpha=(\{X_g\}_{g\in G},\{\alpha_g\}_{g\in G})$ of a group $G$ on a set $X$, we define a map $\hat{\alpha}\colon G\to \End(\C[X])$ by setting
	\[ \hat{\alpha}(g)(x)\coloneqq \alpha_g(P_{X_{g^{-1}}}(x))   \] 
	for all $x\in X$, extended by linearity. 
		It is easy to check that $(\C[X],\hat{\alpha})$ defines indeed a partial representation of $G$ on $\C[X]$, which we call the \emph{linearization} of $(X,\alpha)$.
\end{definition}
\begin{remark}
	Notice that given a partial representation $(V,\pi)$ of $G$, setting $V_{g}\coloneqq \pi(g)\pi(g^{-1})(V)$ and $\alpha_g\coloneqq \left.\pi(g)\right|_{V_{g^{-1}}}\colon V_{g^{-1}}\to V_g$ for every $g\in G$ gives a partial action $\alpha=(\{V_g\}_{g\in G},\{\alpha_g\}_{g\in G})$ of $G$ on $V$.
\end{remark}

\begin{notation}
If there is no risk of confusion, we often denote the vector space generated by a set $X$ simply by $\C X$, without parenthesis. In particular, the group algebra over a group $G$ is denoted indifferently by $\C[G]$ or $\C G$. This is aimed at lightening the notation, for example when considering the groupoid algebra $\C\Gamma_H (G)$.
\end{notation}

We will see more examples of partial representations later in this work.

\medskip

Keeping in mind what we saw for partial actions, it is natural to ask if there is a way to restrict a global representation of a group $G$ on a vector space $U$ to a partial representation of $G$ on a proper subspace of $U$. A minute of thought suggests that we need some more information to do this. We propose the following definition.

\begin{definition} \label{def:restriction_Gglobal}
	Let $(U,\rho)$ be a global representation of a group $G$ on a vector space $U$, and let $\varphi\colon V\to U$ and $\tau\colon U\to V$ be two linear maps such that $\tau\circ \varphi=\Id_V$. Consider the map $\pi\colon G\to \End(V)$ defined by
	\begin{equation}\label{eq:induced}
	\pi(g)(v)\coloneqq \tau(\rho(g)(\varphi(v)))
	\end{equation}
	for all $v\in V$, $g\in G$. We say that $(V,\pi)$ is the \emph{restriction} of the global representation $(U,\rho)$ to $V$ via $\varphi$ and $\tau$ if
	\begin{enumerate}[label=(RR\arabic*),ref=(RR\arabic*)]
		\item\label{item:RR1} $(V,\pi)$ is a partial representation of $G$;
		\item\label{item:RR2} for every $g\in G$ and $v\in V_{g^{-1}}\coloneqq \pi(g^{-1})\pi(g)(V)$ we have
		\begin{equation}\label{eq:varphipi} 
		\varphi(\pi(g)(v))=\rho(g)(\varphi(v)). 
		\end{equation}
	\end{enumerate}
\end{definition}

\begin{remark} \label{rmk:ccondition}
	Notice that in \cite{Alves_Batista_Vercruysse-dilations} (following \cite{Abadie_dilations}) a restriction of a partial representation $U$ to a subspace $V$ is defined in terms of a \emph{$c$-condition} on a projection $T\colon U\to U$ with $T(U)=V$. It is easy to see that this condition for $T\coloneqq \varphi\circ \tau$ implies our \ref{item:RR1} and \ref{item:RR2}, but the converse does not seem to hold in general. In fact it can be shown that \ref{item:RR1} and \ref{item:RR2} are equivalent to a \emph{restricted $c$-condition}, which amounts to require that the \emph{$c$-condition} holds on $T(U)$ instead of on the whole $U$. 
\end{remark}

The first example is of course coming from the restrictions of global actions to subsets.
\begin{example} \label{ex:linear_restriction}
	Let $(Y,\beta)$ be a global action of a group $G$ on the set $Y$. Let $X\subseteq Y$ be a subset, and let $(X,\alpha)$ be the restriction of $(Y,\beta)$ to $X$. Let us check that the linearization $(\C[X],\hat{\alpha})$ of $(X,\alpha)$ is indeed the restriction of the linearization $(\C[Y],\hat{\beta})$ of $(Y,\beta)$ via the inclusion $\varphi\colon \C[X]\to \C[Y]$ and the projection $\tau\coloneqq P_{X}\colon\C[Y]\to \C[X]$. Fix $g\in G$. For all $x\in X$, 
	\[
	\hat{\alpha}(g)(x) = \alpha_g(P_{X_{g^{-1}}}(x)) = \begin{cases} 0 & \text{if }x\notin X_{g^{-1}} \\ \alpha_g(x) & \text{if }x\in X_{g^{-1}} \end{cases} \stackrel{\ref{item:R2}}{=} \begin{cases} 0 & \text{if }x\notin X_{g^{-1}} \\ \beta_g(x) & \text{if }x\in X_{g^{-1}} \end{cases}.
	\]
	On the other hand,
	\[
	\tau(\hat{\beta}(g)(\varphi(x))) = \tau(\beta_g(x)) = \begin{cases} 0 & \beta_g(x)\notin X \\ \beta_g(x) & \beta_g(x)\in X \end{cases}.
	\]
	However, $\beta_g(x)\in X$ if and only if $x\in X \cap \beta_g^{-1}(X) \stackrel{\ref{item:R1}}{=} X_{g^{-1}}$. Therefore,
	\begin{equation*}
	\hat{\alpha}(g) = \tau \circ \hat{\beta}(g) \circ \varphi \qquad \text{for all } g\in G
	\end{equation*}
	and the map $\pi \colon G \to \End(\C[X])$ given by $\pi(g) \coloneqq \tau \circ \hat{\beta}(g) \circ \varphi$ for all $g\in G$ is a partial representation. Moreover, for every $g\in G$ we have
	\[
	\C[X]_{g^{-1}} = \pi(g^{-1})\pi(g)(\C[X]) = \pi(g^{-1})\left(\C\left[X\cap \beta_g(X)\right]\right) = \C\left[\beta_{g^{-1}}(X)\cap X\right] \stackrel{\ref{item:R1}}{=} \C\left[X_{g^{-1}}\right]
	\]
	and for every $x\in X_{g^{-1}}$,
	$
	\varphi(\pi(g)(x)) = \beta_g(x) = \hat{\beta}(g)(\varphi(x)),
	$
	so that $(\C[X],\pi)=(\C[X],\hat{\alpha})$ is indeed the restriction of $\hat{\beta}$ to $\C[X]$ via $\varphi$ and $\tau$. 
\end{example}

It turns out that all partial representations can be obtained as restrictions of global representations, as we will prove in the next section.

\subsection{Globalization} \label{sec:globalization}

In analogy to the case of partial actions, we propose the following definition.

\begin{definition}\label{def:globalization}
	A \emph{globalization} of a partial representation $(V,\pi)$ of a group $G$ is a quadruple $(U,\rho,\varphi,\tau)$ where
	\begin{enumerate}[label=(GR\arabic*),ref=(GR\arabic*)]
		\item\label{item:GR1} $(U,\rho)$ is a global representation of $G$;
		\item\label{item:GR2} $(V,\pi)$ is the restriction of $(U,\rho)$ via $\varphi$ and $\tau$;
		\item\label{item:GR3} for every quadruple $(U',\rho',\varphi',\tau')$ satisfying \ref{item:GR1} and \ref{item:GR2} there exists a unique  $G$-homomorphism $\psi\colon U\to U'$ (i.e.\ $\psi$ is linear and $\psi\circ \rho(g)=\rho'(g)\circ \psi$ for all $g\in G$) such that $\psi\circ \varphi=\varphi'$ and $\tau'\circ \psi= \tau$.
	\end{enumerate}
\end{definition}

In the next theorem we show the existence and uniqueness up to isomorphism of globalizations of $G$-partial representations. By Remark~\ref{rmk:ccondition}, this result extends the ones in \cite{Alves_Batista_Vercruysse-dilations} and \cite{Abadie_dilations} (cf.\ also \cite{Saracco_Vercruysse-Globalization}), giving a construction that is closer to the one for partial actions in \cite{Abadie-globalization-2003}.

\begin{theorem} \label{thm:globalization}
	Every $G$-partial representation $(V,\pi)$ has a unique globalization $(U,\rho,\varphi,\tau)$ up to a canonical isomorphism, i.e.\ if $(U',\rho',\varphi',\tau')$ is another globalization, then there exists a unique $G$-isomorphism $\psi\colon U\to U'$ such that $\psi\circ \varphi=\varphi'$ and $\tau'\circ \psi=\tau$. 
\end{theorem}

\begin{proof}
	Let $(V,\pi)$ be a $G$-partial representation. Set $V_g \coloneqq \pi(g)\pi(g^{-1})(V)$ for all $g\in G$. Consider, in the complex vector space $\mathbb{C}[G]\otimes V$, the subspace $Z$ generated by the vectors 
	\begin{equation}\label{eq:Z}
	\{g\otimes v-h\otimes \pi(h^{-1}g)(v)\mid g,h\in G, v\in V_{g^{-1}h}\}
	\end{equation}
	and let $U$ be the quotient space $(\mathbb{C}[G]\otimes V)/Z$. Recall that the left multiplication on the first tensorand makes of $\mathbb{C}[G]\otimes V$ a $G$-global representation. This induces a structure of $G$-global representation $(U,\rho)$ on $U$ by setting
	\[
	\rho(g)(\overline{h\otimes v})\coloneqq \overline{gh\otimes v} 
	\]
	for all $g,h\in G$ and $v\in V$, where $\overline{t}$ denotes the coset of $t\in \mathbb{C}[G]\otimes V$ in $U$.
	To check that $\rho$ is well defined, it is enough to notice that $Z$ is a $G$-subrepresentation: given $g,h,k\in G$ and $v\in V_{h^{-1}k}$, we have
	\[ 
	gh \otimes v - gk \otimes \pi(k^{-1}h)(v) = gh \otimes v - gk \otimes \pi(k^{-1}g^{-1}gh)(v) = (gh) \otimes v - (gk) \otimes \pi\left((gk)^{-1}(gh)\right)(v) \in Z
	\]
	which proves the claim.
	
	Clearly the maps $\rho(g)$ are invertible, as $\rho(g)^{-1}=\rho(g^{-1})$, so that $\rho$ gives indeed a global representation of $G$. 
	
	Consider now the map $\varphi\colon V\to U$ defined by $\varphi(v)\coloneqq \overline{1_G\otimes v}$ for all $v\in V$, and the map $\tau:U\to V$ defined by $\tau(\overline{g\otimes v})\coloneqq \pi(g)(v)$ for all $v\in V$. Observe that the latter is well defined, since the map $\widetilde{\tau}\colon \mathbb{C}[G]\otimes V\to V$ defined by $\widetilde{\tau}(g\otimes v)\coloneqq \pi(g)(v)$ sends the generators of $Z$ to $0$: for $h,k\in G$ and $v\in V_{h^{-1}k}=\pi(h^{-1}k)\pi(k^{-1}h)(V)$, we have 
	\begin{align*} 
		\widetilde{\tau}(h\otimes v) & \stackrel{\phantom{\ref{item:PR2}}}{=}  \pi(h)(v) \stackrel{\eqref{eq:projections}}{=} \pi(h)\pi(h^{-1}k)\pi(k^{-1}h) (v) \stackrel{\ref{item:PR2}}{=} \pi(h)\pi(h^{-1})\pi(k)\pi(k^{-1}h) (v)\\
		& \stackrel{\ref{item:PR2}}{=}  \pi(h)\pi(h^{-1})\pi(k)\pi(k^{-1})\pi(h) (v) \stackrel{\eqref{eq:commidem}}{=} \pi(k)\pi(k^{-1})\pi(h)\pi(h^{-1})\pi(h) (v)\\
		& \stackrel{\ref{item:PR2}}{=}  \pi(k)\pi(k^{-1})\pi(h) (v) \stackrel{\ref{item:PR2}}{=} \pi(k)\pi(k^{-1}h) (v) = \widetilde{\tau}(k\otimes \pi(k^{-1}h)(v))
	\end{align*}
	and so $\tau$ is well defined.
	
	We want to show that $(U,\rho,\varphi,\tau)$ is a globalization of $(V,\pi)$. We already showed that $(U,\rho)$ is a global representation of $G$, proving \ref{item:GR1}. Moreover, it is clear that $\tau\circ \varphi=\Id_V$ and, by construction, $\tau(\rho(g)(\varphi(v))=\pi(g)(v)$. To prove the remaining property of \ref{item:GR2}, observe that for each $v\in V_{g^{-1}}$, \[ \varphi(\pi(g)(v))=\overline{1_G\otimes \pi(g)(v)} = \overline{g\otimes v}=\rho(g)(\overline{1_G\otimes v})=\rho(g)(\varphi(v)).  \]
	Therefore $(V,\pi)$ is the restriction of $(U,\rho)$ via $\varphi$ and $\tau$, proving \ref{item:GR2}.
	
	In order to prove \ref{item:GR3}, let $(U',\rho',\varphi',\tau')$ be another quadruple satisfying \ref{item:GR1} and \ref{item:GR2}. Since
	\begin{equation} \label{eq:obvious}
	U=\sum_{g \in G}\rho(g)(\varphi(V)),
	\end{equation}
	if a $\psi\colon U\to U'$ with the properties stated in \ref{item:GR3} exists, then it is uniquely determined by the property $\psi(\varphi(v))=\varphi'(v)$ for all $v\in V$, which gives $\psi(\overline{1_G\otimes v})=\psi(\varphi(v))=\varphi'(v)$, and the property $\psi(\rho(g)(u))=\rho'(g)(\psi(u))$ for all $u\in U$ and $g\in G$, so that \[\psi(\overline{g\otimes v})=\psi(\rho(g)(\overline{1_G\otimes v}))=\psi(\rho(g)(\varphi(v))) =\rho'(g)(\psi(\varphi(v)))=\rho'(g)(\varphi'(v)).\]
	This shows the uniqueness of such a map $\psi$, and it suggests how to define it: we define $\widetilde{\psi}\colon \C[G]\otimes V\to U'$ by setting $\widetilde{\psi}(g\otimes v)\coloneqq \rho'(g)(\varphi'(v))$ for all $g\in G$ and $v\in V$. This map sends the generators of $Z$ to $0$: indeed if $g,h\in G$ and $v\in V_{g^{-1}h}$,
	\[ \widetilde{\psi}(g\otimes v)= \rho'(g)(\varphi'(v))=\rho'(h)\rho'(h^{-1}g)(\varphi'(v)) \stackrel{\ref{item:RR2}}{=} \rho'(h)(\varphi'(\pi'(h^{-1}g)(v)))= \widetilde{\psi}(h\otimes \pi'(h^{-1}g)(v))\]
	and hence the map $\psi\colon U\to U'$ given by $\psi(\overline{g\otimes v})\coloneqq \rho'(g)(\varphi'(v))$ is well-defined. Notice that, by construction, $\psi$ is a $G$-homomorphism and $\psi\circ \varphi=\varphi'$. 
	
	The property $\tau'\circ \psi= \tau$ follows easily from
	\[ \tau(\rho(g)(\varphi(v)))=\pi(g)(v)=\tau'(\rho'(g)(\varphi'(v)))= \tau'(\rho'(g)(\psi(\varphi(v))))= \tau'(\psi(\rho(g)(\varphi(v)))) \]
	and the obvious \eqref{eq:obvious}.	
	This shows the existence of $\psi$, completing the proof of \ref{item:GR3}.
	
	Therefore, a globalization exists. Its uniqueness now follows easily from the universal property \ref{item:GR3}.
\end{proof}
\begin{remark} \label{rem:primeaxioms}
	It follows from the proof of Theorem \ref{thm:globalization} that property \ref{item:GR3} of our definition of globalization can be replaced by the following two properties:
	\begin{enumerate}[label=(GR\arabic*'),start=3,ref=(GR\arabic*')]
		\item\label{item:GR3'} $U=\sum_{g\in G}\rho(g)(\varphi(V))$;
		\item\label{item:GR4'} for every quadruple $(U',\rho',\varphi',\tau')$ satisfying \ref{item:GR1} and \ref{item:GR2}, the assignment
		\[\psi(\rho(g)(\varphi(v)))\coloneqq \rho'(g)(\varphi'(v))\qquad \text{for all }g\in G,v\in V \] gives a well-defined linear map $\psi\colon U\to U'$.
	\end{enumerate}
\end{remark}

\begin{example}
	If $\alpha=(\{X_g\}_{g\in G},\{\alpha_g\}_{g\in G})$ is a partial action of a group $G$ on a set $X$, then by Abadie's Theorem~\ref{thm:abadie} there exists a (unique up to isomorphism) globalization $(Y,\beta,\varphi)$ of $\alpha$. It is easy to check that the linearization $(\C[Y],\hat{\beta})$ of the global action of $G$ on $Y$ together with the obvious linearization $\hat{\varphi}\colon \C[X]\to \C[Y]$ of $\varphi$ and the projection $\hat{\tau}\coloneqq P_X\colon \C[Y]\to \C[X]$ gives a globalization of the linearization $(\C[X],\hat{\alpha})$ of $\alpha$. Indeed the properties \ref{item:GR1} and \ref{item:GR2} have been already discussed in Example \ref{ex:linear_restriction}. Instead of \ref{item:GR3}, by Remark~\ref{rem:primeaxioms}, we will check \ref{item:GR3'} and \ref{item:GR4'}. Observe that property \ref{item:GL5} of $(Y,\beta,\varphi)$ implies property \ref{item:GR3'} of its linearization: it is enough to proceed by double inclusion, after noticing that $\hat{\beta}(g)\left(\hat{\varphi}(x)\right) = \beta_g(\varphi(x))$ for all $x \in X$, $g \in G$. Now in order to check \ref{item:GR4'}, let us show that $Y'\coloneqq \bigcup_{g\in G}\bigcup_{x\in X}\rho'(g)(\varphi'(x))\subseteq U'$ together with the restrictions of $\varphi'$ to $X$ and of $\rho'(g)$ to $Y'$ for all $g\in G$ form a triple satisfying \ref{item:GL1}-\ref{item:GL4}. Only \ref{item:GL3} is not immediate. Pick $x \in X_g$ and set $y \coloneqq \alpha_{g^{-1}}(x) \in X_{g^{-1}}\subseteq \C[X]_{g^{-1}}$. It follows that
	\[
	\varphi'(x) = \varphi'(\alpha_g(y)) = \varphi'(\hat{\alpha}(g)(y)) \stackrel{\ref{item:RR2}}{=} \rho'(g)(\varphi'(y)) \in \rho'(g)(\varphi'(X))
	\]
	and hence $\varphi'(X_g) \subseteq \varphi'(X) \cap \rho'(g)(\varphi'(X))$. For the reverse inclusion, pick $z \in X$ such that $\varphi'(z) = \rho'(g)(\varphi'(y))$ for a certain $y \in X$. Then
	\[
	z = \tau'(\varphi'(z)) = \tau'\left(\rho'(g)(\varphi'(y))\right) \stackrel{\eqref{eq:induced}}{=} \hat{\alpha}(g)(y) = \alpha_g(y) \in X_{g}
	\]
	and hence $\varphi'(z) \in \varphi'(X_{g})$. Therefore, by Remark~\ref{rem:globuniversal}, the map $\psi\colon Y\to Y'$ given by $\psi(\beta_g(\varphi(x)))\coloneqq \rho'(g)(\varphi'(x))$ for all $x\in X$ and $g\in G$ is well-defined and its linearization $\hat{\psi}\colon \C[Y]\to U'$ is the map required in \ref{item:GR4'}. This completes the proof that $(\C[Y],\hat{\beta})$ is indeed the globalization of $(\C[X],\hat{\alpha})$.
\end{example}
We will see more examples later in this article.

\section{$H$-global $G$-partial representations}

\noindent In this section we introduce $H$-global $G$-partial representations and we describe some of their general properties.

\subsection{Basic definitions and examples} \label{sec:HGbasics}

We introduce some basic notions and examples, setting up the framework in which we are going to work for the rest of the present article. 

\medskip

\begin{definition}
	Let $H$ be a subgroup of a group $G$.
	
	A $G$-partial action $\alpha=(\{X_g\}_{g\in G},\{\alpha_g\}_{g\in G})$ on a set $X$ is called \emph{$H$-global} if $X_h=X$ for all $h\in H$.
	
	A $G$-partial representation $(V,\pi)$ is \emph{$H$-global} if the restriction of $\pi$ to $H$ is a global representation of $H$.
\end{definition}

\begin{example} \label{ex:trivialHglGpar}
	Let $G$ be a group and $H$ be a subgroup of $G$. There are two kinds of $H$-global $G$-partial representations that are always available.
	\begin{enumerate}[label=(\alph*)]
		
		\item Any $G$-global representation is obviously $H$-global $G$-partial. 
		
		\item Let $(W,\rho)$ be a global representation of $H$. Then we can always construct an $H$-global $G$-partial representation $(W,\overline{\rho})$ in the following way: we define
		\[ \overline{\rho}(g)\coloneqq \left\{\begin{array}{ll} 
		0 & \text{ if }g\in G\setminus H\\
		\rho(g) & \text{ if }g\in H 
		\end{array}\right. . \]
		It is straightforward to check that this is indeed an $H$-global $G$-partial representation.
		
	\end{enumerate}
\end{example}

Before seeing more examples, we introduce one more definition.

\begin{definition}
	Let $G$ be a group and let $(V,\pi)$ be a $G$-partial representation. We define the \emph{globalizer} $H(V)=H(V,\pi)$ of $V$ to be
	\[ H(V)\coloneqq \{g\in G\mid \pi(g)\text{ is invertible} \}.  \]
\end{definition}

\begin{remark}
	Let $(V,\pi)$ be a $G$-partial representation. Since for all $g \in H(V)$
	\[
	\pi(g)\pi(g)^{-1}\pi(g) = \pi(g) \stackrel{\ref{item:PR2}}{=} \pi(g)\pi(g^{-1}) \pi(g) \qquad \Longrightarrow \qquad \pi(g)^{-1} = \pi(g^{-1}),
	\]
	it is straightforward to check that $H(V)$ is a subgroup of $G$: it is the biggest subgroup of $G$ that acts globally on $V$ via $\pi$, i.e.\ $(V,\pi)$ is $H$-global if and only if $H$ is a subgroup of $H(V)$.
\end{remark}

It turns out that the globalizer $H(V)$ of a $G$-partial representation $(V,\pi)$ is typically nontrivial. This is one of the main motivations to study $H$-global $G$-partial representations.

\begin{example} \label{ex:Hglobal_action}
	Let $K$ be a subgroup of a group $G$ and consider the action of $G$ on the left cosets $G/K$ of $K$ by left multiplication. Let $A\subseteq G/K$ be a set of left cosets of $K$ and consider the restriction $\alpha$ of the given action to $A$, which gives us a $G$-partial action. Set \[ H(A)\coloneqq \{g\in G\mid gA=A\}. \] It is easy to check the $G$-partial action $\alpha$ is $H$-global if and only if $H$ is a subgroup of $H(A)$.
	
	Consider now the linearization $(\C[A],\hat{\alpha})$ of $\alpha$. It is clear that the globalizer of this $G$-partial representation is in fact $H(\C[A])=H(A)$.
\end{example}
The following example is a special case of the previous one; it will be relevant in Section~\ref{sec:SnSnminus1}.
\begin{example} \label{ex:SkSnk_action}
	Let us fix some notation. For any positive integer $n\in  \mathbb{N}$, let $[n]:=\{1,2,\dots,n\}$, and for $I\subseteq [n]$, let
	\[ \mathfrak{S}_{n}^I\coloneqq \{\sigma\in \mathfrak{S}_n\mid \sigma(I)=I\},\] 
	so for example $\mathfrak{S}_{n}^{\varnothing}=\mathfrak{S}_{n}^{[n]}=\mathfrak{S}_{n}$ and $\mathfrak{S}_{n}^{\{1\}}=\mathfrak{S}_{n}^{[1]}=\mathfrak{S}_{n}^{[n]\setminus\{1\}}\cong \mathfrak{S}_{1}\times \mathfrak{S}_{n-1} \equiv \mathfrak{S}_{n-1}\subseteq \mathfrak{S}_{n}$. 
	
	More generally, for $1\leq k< n$, consider $\mathfrak{S}_{k}\times \mathfrak{S}_{n-k}\equiv \mathfrak{S}_{n}^{[k]}=\mathfrak{S}_{n}^{[n]\setminus [k]}$. We want to describe the action of $\mathfrak{S}_n$ on the left cosets $\mathfrak{S}_n/(\mathfrak{S}_{k}\times \mathfrak{S}_{n-k})$.
	
	Observe that, given $\sigma \in \mathfrak{S}_{n}$ and $I\subseteq [n]$,
	\begin{align*}
	\sigma   \mathfrak{S}_{n}^{I} & =\{\tau\in \mathfrak{S}_n\mid \tau(I)=\sigma(I) \}
	\end{align*}
	therefore, for $\rho \in \mathfrak{S}_{n}$,
	\[
	\rho\cdot \left(\sigma   \mathfrak{S}_{n}^{I}\right) = \tau   \mathfrak{S}_{n}^{I} \quad \iff \quad \rho\sigma(I) = \tau(I),
	\]
	if and only if $\rho$ maps the subset $\sigma(I)\subseteq [n]$ into the subset $\tau(I)\subseteq [n]$ and hence we can identify the action of $\mathfrak{S}_{n}$ on the left cosets of $\mathfrak{S}_{k}\times \mathfrak{S}_{n-k}\equiv \mathfrak{S}_{n}^{[k]}$ with the action of $\mathfrak{S}_{n}$ on 
	\[ \binom{[n]}{k}\coloneqq \{A\subseteq [n]\mid |A|=k \} \]
	via $\sigma\Sym_n^{[k]} \leftrightarrow \sigma([k])$. We will freely use this identification in the rest of this article.
	
	Notice that for $k=1$ we are identifying the action of $\mathfrak{S}_n$ on the cosets $\mathfrak{S}_n/\mathfrak{S}_{n-1}$ with the defining action of $\Sym_n$ on the set $\binom{[n]}{1}=[n]$.
	
	We now set 
	\[Y\coloneqq \left\{\left.A\in \binom{[n]}{k}\ \right|\ 1\in A \right\}\subsetneq \binom{[n]}{k}\]
	and call $\alpha$ the $\Sym_n$-partial action on $Y$ that we get by restriction from the $\Sym_n$-global action on $\binom{[n]}{k}$. For instance, for $n=4$ and $k=2$,
	\[
	\binom{[4]}{2} = \left\{\{1,2\},\{1,3\},\{1,4\},\{2,3\},\{2,4\},\{3,4\}\right\}, \qquad Y = \left\{\{1,2\},\{1,3\},\{1,4\}\right\}
	\]
	and, for example,
	\begin{gather*}
	Y_{(1,2)} = Y \cap (1,2)\cdot Y = \left\{\{1,2\},\{1,3\},\{1,4\}\right\} \cap \left\{\{1,2\},\{2,3\},\{2,4\}\right\} = \{(1,2)\}\\
	\text{while} \qquad Y_{(2,3,4)} = \left\{\{1,2\},\{1,3\},\{1,4\}\right\} \cap \left\{\{1,3\},\{1,4\},\{1,2\}\right\} = Y.
	\end{gather*}
	In the notation of Example~\ref{ex:Hglobal_action}, it is easy to check that $H(Y)=\Sym_n^{[1]}$, so that the globalizer of the linearization $(\C[Y],\hat{\alpha})$ of this $\Sym_n$-partial action is $\Sym_n^{[1]}$. 
\end{example}
As we develop the theory, we will see a lot of additional examples.

\subsection{A general construction} \label{sec:genconstruction}

In this section we give a general construction that gives a large class of examples of $H$-global $G$-partial representations.

\medskip

We start with a finite group $G$ and two fixed subgroups $H$ and $K$ of $G$. Let $G=\coprod_{i=1}^ng_iK$.

Let $A\subseteq G$ be a union of $(H,K)$-double cosets, i.e.\ $hA=A$ for all $h\in H$ and $Ak=A$ for all $k\in K$. Denote by $A/K$ the set of left cosets of $K$ contained in $A$. Let $(W,\rho)$ be a global representation of $K$. With these data we want to construct an $H$-global $G$-partial representation.

Recall that the induced representation $\mathsf{Ind}_K^GW=\C[G]\otimes_{\mathbb{C}[K]}W$ decomposes as a vector space as
\[ \mathsf{Ind}_K^GW\cong \bigoplus_{g_iK\in G/K}W^{g_i}, \]
where we denoted by $W^{g_i}$ the subspace $\C g_iK\otimes_{\mathbb{C}[K]}W$ corresponding to the coset $g_iK$. For every $i=1,\ldots,n$, we call $\phi_i\colon W\to W^{g_i}$ the $\C$-linear isomorphism defined by $\phi_i(w)\coloneqq g_i\otimes_{\mathbb{C}[K]} w$ for every $w\in W$.

Consider the subspace
\[ W_A:= \bigoplus_{g_iK\in A/K}W^{g_i}.\]

We define a linear map $\pi\colon G\to \End(W_A)$ in the following way: for every $g_iK\in A/K$ and every $x\in W^{g_i}$, there exists a unique $w \in W$ such that $x = \phi_i(w)$. Thus, we set
\begin{equation}\label{s1eq:GenParRep}
\pi(g)(x)\coloneqq \left\{ \begin{array}{ll}
\phi_j(\rho(k)(w)) & \text{ if }x=\phi_i(w),gg_i=g_jk,k\in K,\text{ and }g_jK\in A/K\\
\qquad 0 & \text{ otherwise }
\end{array} \right. .
\end{equation}
It is easy to check that this definition does not depend on the choice of the representatives of the left cosets of $K$.

\begin{lemma}
	The pair $(W_A,\pi)$ is an $H$-global $G$-partial representation.
\end{lemma}
\begin{proof}
	The proof that $(W_A,\pi)$ is a $G$-partial representation does not require that $A$ satisfies the property that $hA=A$ for all $h\in H$. This property is used only to show that this $G$-partial representation is indeed $H$-global. We leave the tedious but straightforward details to the reader.
\end{proof}
\begin{remark} \label{rem:ourconstr_restriction}
	Observe that the construction of $(W_A,\pi)$ is in fact the restriction (cf. Definition~\ref{def:restriction_Gglobal}) of the global $G$-representation $\Ind_K^GW$ to the subspace $W_A$ via the obvious inclusion $\varphi$ and the projection $\tau\colon \Ind_K^GW\to W_A$ whose kernel is \[\ker\tau=\bigoplus_{g_iK\notin A/K}W^{g_i}.\]
\end{remark}

\begin{example}
	Using the same notation as before, let $W$ be the trivial representation of $K$. Then in this case $\Ind_K^GW$ is simply the linearization of the action of $G$ on the left cosets $G/K$ of $K$ by left multiplication. Consider a subset $A\subseteq G$ such that $Ak=A$ for all $k\in K$. Then our construction $(W_A,\pi)$ in this case corresponds simply to the linearization of the restriction of this action to $A/K$ (cf.\ Example~\ref{ex:linear_restriction}).
\end{example}
We will see in Section~\ref{sec:irreducibles} that all the irreducible objects in the category of $H$-global $G$-partial representations can be built with the construction given in the present section.

\subsection{Some general properties of partial representations}\label{sec:GeneralProperties}

In this section we discuss some general properties of $G$-partial representations. In particular we introduce certain orthogonal idempotents that give a general decomposition of such representations. Then we specialize this discussion to the case of $H$-global $G$-partial representations, to see what this decomposition looks like in this case. 

\medskip

All over this section we consider a partial representation $(V,\pi)$ of a finite group $G$.
\begin{notation} 
	Set
	\[\mathcal{P}_{H}(G/H)\coloneqq \{A\subseteq G\mid H\subseteq A\text{ and }A\text{ is a union of left cosets of }H  \}. \]
	In the case \(H = \{1_G\}\), we use the notation
	\[
	\mathcal{P}_{1}(G)\coloneqq \mathcal{P}_{\{1_G\}}(G/\{1_G\})=\{A\subseteq G\mid 1_G\in A\}.
	\]
\end{notation}

\begin{definition}
	\label{PA}
	For every $A\subseteq G$ define the operator
	\[ P_A^\pi=P_A:=\prod_{g\in A}\pi(g)\pi(g^{-1})\cdot \prod_{\overline{g}\in G\setminus A}(\Id_V-\pi(\overline{g})\pi(\overline{g}^{-1})). \]
	Observe that, thanks to Lemma~\ref{lem:idemp}, in the definition of $P_A^\pi$ we do not need to specify the order of the products.
\end{definition}

\begin{remark} \label{rem:P_Azero}
	Notice that if $1_G\notin A$, i.e.\ $A\notin \mathcal{P}_1(G)$, then $P_A$ is the zero operator, as it contains the factor $\Id_V-\pi(1_G)\pi(1_G)=0$.
\end{remark}
\begin{lemma}
	For any $A \subseteq G$ and any $g\in G$,
	\begin{equation}\label{s3eq:PAh}
	\pi(g)P_A = P_{gA} \pi(g).
	\end{equation}
	In particular, if $g^{-1}\notin A$, then 
	\begin{equation} \label{eq:gP_A=0}
	\pi(g)P_A=0=P_A\pi(g^{-1}).
	\end{equation}
\end{lemma}
\begin{proof}
	The first claim follows immediately from \eqref{eq:idemp}, i.e.\ $\pi(g)\pi(h)\pi(h^{-1})=\pi(gh)\pi(h^{-1}g^{-1})\pi(g)$ for all $g,h\in G$. The second claim follows from the first one and from Remark~\ref{rem:P_Azero}.
\end{proof}

\begin{lemma}
	Consider $A\subseteq G$. If $g \in A$, then\[\pi(g)\pi(g^{-1})P_A = P_A\pi(g)\pi(g^{-1})=P_A.\] If \(g \notin A\), then \[\pi(g)\pi(g^{-1})P_A = P_A\pi(g)\pi(g^{-1})= 0.\]
\end{lemma}
\begin{proof}
	It follows immediately from the definition of $P_A$, Lemma~\ref{lem:idemp} and equation~\eqref{eq:gP_A=0}.
\end{proof}

\begin{lemma} \label{lem:sumP_AequalId}
	The set $\{P_A\}_{A\in \mathcal{P}_{1}(G)}$ is a system of orthogonal idempotents. Moreover, $\sum_{A\in \mathcal{P}_{1}(G)}P_A=\Id_V$.
\end{lemma}
\begin{proof}
	The fact that every \(P_A\) is idempotent follows from Lemma~\ref{lem:idemp}, and it is easy to see that \(P_A P_B = P_B P_A = 0\) if \(A \neq B\). For the last statement, we compute
	\begin{align*}
	\Id_V &= \prod_{g \in G} \Id_V = \prod_{g \in G} (\Id_V - \pi(g) \pi(g^{-1}) + \pi(g) \pi(g^{-1})) \\
	&= \sum_{A \subseteq G} \left(\prod_{g \in A} \pi(g) \pi(g^{-1})\right) \left(\prod_{g \notin A} (\Id_V - \pi(g) \pi(g^{-1}))\right) \\
	&= \sum_{A \subseteq G: 1_G \in A} \left(\prod_{g \in A} \pi(g) \pi(g^{-1})\right) \left(\prod_{g \notin A} (\Id_V - \pi(g) \pi(g^{-1}))\right) \\
	&= \sum_{A \in \mathcal{P}_1(G)} P_A. \qedhere
	\end{align*}
\end{proof}

As a consequence of the previous lemma, if we set $V^A:=P_A^\pi(V)$, then we have the decomposition
\begin{equation} \label{eq:decomp}
\quad V=\bigoplus_{A\in \mathcal{P}_{1}(G)}V^A.
\end{equation}
Notice that, as observed above, if $g^{-1} \notin A$, then $\pi(g)(V^A)=0$.

Observe also that a $G$-homomorphism $\varphi\in \mathsf{Hom}_G(V,U)$, i.e.\ a morphism between the two $G$-partial representations $(V,\pi)$ and $(U,\eta)$, intertwines the action of $P_A^\pi$ and $P_A^\eta$, i.e. 
\[ P_A^\eta (\varphi(v))=\varphi(P_A^\pi(v))\quad \text{for all }v\in V, \]
so that, if we set $V^A\coloneqq P_A^\pi(V)$ and $U^A\coloneqq P_A^\eta(U)$, then we have the decompositions
\[ V=\bigoplus_{A\in \mathcal{P}_{1}(G)}V^A \quad \text{ and } \quad U=\bigoplus_{A\in \mathcal{P}_{1}(G)}U^A\]
and $\varphi$ is a graded map, i.e.\ $\varphi(V^A)\subseteq U^A$ for all $A$.

\medskip

We want to understand this decomposition in the case when the $G$-partial representation $(V,\pi)$ is $H$-global, for a subgroup $H$ of $G$, so we assume this from now on.

The following easy observation is so useful that deserves to be called a lemma.
\begin{lemma}\label{s6lemma:H-lin}
	Let \((V,\pi)\) be an \(H\)-global \(G\)-partial representation. Then for all $g\in G$ and $h\in H$ we have
	\begin{equation}  \label{eq:BasicLemma}
	\pi(gh)=\pi(g)\pi(h) \quad \text{ and } \quad \pi(hg)=\pi(h)\pi(g). 
	\end{equation}
\end{lemma}
\begin{proof}
	It follows from the defining properties of an $H$-global $G$-partial representation that
	\[ \pi(gh)=\pi(gh)\pi(1_G)=\pi(gh)\pi(h^{-1})\pi(h)= \pi(g)\pi(h). \]
	The proof of $\pi(hg)=\pi(h)\pi(g)$ is similar.
\end{proof}

\begin{lemma}
	\label{PAhisPA}
	Let \((V,\pi)\) be an \(H\)-global \(G\)-partial representation. Then \(P_{Ah}^\pi = P_A^\pi\) for any $A \in \mathcal{P}_1(G)$ and $h \in H$.
\end{lemma}
\begin{proof}
	Using \eqref{eq:BasicLemma} this is a straightforward verification.
\end{proof}

\begin{corollary}
	\label{PAnonzero}
	If $A \in \mathcal{P}_1(G)$ such that \(P_A \neq 0\), then \(A\) is a union of left cosets of \(H\), and $H\subseteq A$.
\end{corollary}
\begin{proof}
	Since the $P_A$ are orthogonal idempotents, if $Ah\neq A$ for some $h\in H$, then $P_{Ah}=P_A$ implies $P_A=P_{Ah}=0$. In other words, for $P_A$ to be nonzero we must have $Ah=A$ for all $h\in H$, i.e.\ $A$ is a union of left cosets of $H$. Since $1_G\in A$, it is clear that $H\subseteq A$.
\end{proof}
As a consequence of this corollary, the decomposition \eqref{eq:decomp} of an $H$-global $G$-partial representation $(V,\pi)$ simplifies to
\[ \quad V=\bigoplus_{A\in \mathcal{P}_{H}(G/H)}V^A.\]
In particular, if $w\in V^A=P_A(V)$, then for any $g\in G$
\[ \pi(g)(w) \stackrel{\phantom{(6.1)}}{=} \pi(g)P_A(w) \stackrel{\eqref{s3eq:PAh}}{=} P_{gA}\pi(g)(w),  \]
so that \[\pi(g)(w)\in V^{gA}=P_{gA}(V).\] 

The following results are two interesting and useful consequences of Lemma \ref{s6lemma:H-lin} that considerably simplify the computations when dealing with $H$-global $G$-partial representations and the projections $P_A$. In particular, the following lemma gives an alternative characterization of $H$-global $G$-partial representations.

\begin{lemma} \label{lem:HglobalCharacterization}
	A pair \((V,\pi)\), with $V$ a vector space and $\pi\colon G\to \End(V)$, is an $H$-global \(G\)-partial representation if and only if
	\begin{enumerate}[label=(GPR\arabic*),ref=\emph{(GPR\arabic*)},leftmargin=1.7cm]
		\item\label{item:GPR1} $\pi(1_G)=\Id_V$;
		\item\label{item:GPR2} $\pi(\bar{g}^{-1}) \pi(g) \pi(h) = \pi(\bar{g}^{-1}) \pi(gh)$ for any \(\bar{g}, g, h \in G\) such that \(g^{-1}\bar{g} \in H\);
		\item\label{item:GPR3} $\pi(g^{-1}) \pi(h^{-1}) \pi(\bar{h}) = \pi(g^{-1} h^{-1}) \pi(\bar{h})$ for any \(g, h, \bar{h} \in G\) such that \(h^{-1}\bar{h} \in H\).
	\end{enumerate}
	In particular, for an $H$-global \(G\)-partial representation conditions \ref{item:PR2} and \ref{item:PR3} of the definition of a $G$-partial representation are satisfied by taking as ``witness'' any element in the same coset.
\end{lemma}
\begin{proof}
	If \((V,\pi)\) satisfies the stated identities, then clearly it is a $G$-partial representation. To see that it is $H$-global, just notice that if $g,h\in H$, then you can choose $\bar{g}=\bar{h}=1_G$.
	
	Conversely, by using Lemma~\ref{s6lemma:H-lin}, if \((V,\pi)\) is an $H$-global $G$-partial representation, then for \(\bar{g},g,h\in G\) with $\bar{g}^{-1}g\in H$ we have 
	\begin{gather*}
	\pi(\bar{g}^{-1}) \pi(g) \pi(h) = \pi(\bar{g}^{-1}gg^{-1}) \pi(g) \pi(h) \stackrel{\ref{item:PR3}}{=} \pi(\bar{g}^{-1}g)\pi(g^{-1}) \pi(g) \pi(h) \\ 
	\stackrel{\ref{item:PR2}}{=} \pi(\bar{g}^{-1}g)\pi(g^{-1}) \pi(gh) \stackrel{\eqref{eq:BasicLemma}}{=} \pi(\bar{g}^{-1}gg^{-1}) \pi(gh)= \pi(\bar{g}^{-1}) \pi(gh).
	\end{gather*}
	The other property is shown in an analogous way.
\end{proof}

\begin{proposition}\label{s6prop:PA}
	Let \((V,\pi)\) be an \(H\)-global \(G\)-partial representation. Then the assignment $G \to \End(V): g \mapsto \pi(g)\pi(g^{-1})$ is constant on the left cosets of $H$ in $G$. In particular, if $\{g_1,\ldots,g_r\}$ is a family of representatives of left cosets of $H$ in $G$, that is, if $G=\coprod_{i=1}^r g_iH$, then 
	\begin{equation}\label{s6eq:PA}
	P_A^\pi = \prod_{g_kH \subseteq A}\pi(g_k)\pi(g_k^{-1}) \cdot \prod_{g_iH \subseteq G\setminus A}\left(\Id_V - \pi(g_i)\pi(g_i^{-1})\right)
	\end{equation}
	for every $A \in \Pp_H(G/H)$ and the expression \eqref{s6eq:PA} does not depend on the chosen representatives.
\end{proposition}
\begin{proof}
	In view of Lemma \ref{s6lemma:H-lin}, we have for any $g \in G, h \in H$
	\[
	\pi(gh)\pi((gh)^{-1}) = \pi(gh)\pi(h^{-1}g^{-1}) = \pi(g)\pi(h)\pi(h^{-1})\pi(g^{-1}) = \pi(g)\pi(g^{-1}),
	\]
	thus proving the first claim. The second claim is now evident, since the elements $\pi(g)\pi(g^{-1})$ are commuting idempotents (Lemma \ref{lem:idemp}).
\end{proof}

\section{Representation theory: generalities}

In this section we develop the general theory of $H$-global $G$-partial representations, describing in particular its irreducibles.

\subsection{Partial representations as modules}\label{sec:ParRepMods}

It is a well-known fact that any representation of a group on a vector space corresponds to a module over the associated group algebra. In fact, this correspondence gives rise to an isomorphism of categories. A similar construction has been provided in \cite[page 512]{Dokuchaev-Exel-Piccione} in order to relate partial representations of a group with ordinary modules over a suitable \emph{partial group algebra}. The aim of this section is to extend these results to the case of $H$-global $G$-partial representations by introducing (via generators and relations) an associative unital $\C$-algebra $\C^H_{par}G$ such that the category of $H$-global $G$-partial representations is isomorphic to the category of modules over $\C^H_{par}G$. In the forthcoming Section~\ref{sec:groupoid}, we will give an important realization of such an abstract algebra as the groupoid algebra of a certain groupoid naturally associated to $G$ and $H$. Later on, in Section~\ref{sec:semigroup}, we will give a second realization of this algebra as a semigroup algebra of an inverse semigroup. 

Let $G$ be a group and $H\subseteq G$ be a subgroup. Recall from \cite[Definition 2.4]{Dokuchaev-Exel-Piccione} that we can define the \emph{partial group algebra} $\C_{par}G$ as the associative $\C$-algebra generated by the symbols $\{[g]\mid g\in G\}$ subject to the relations 
\begin{equation} \label{eq:relationsSG}
\begin{aligned}
[1_G][g]& =[g][1_G]=[g]\\
[g^{-1}][g][g'] &  = [g^{-1}][gg']\\
[g'][g][g^{-1}] & = [g'g][g^{-1}]
\end{aligned}
\end{equation}
for all $g,g'\in G$. Notice that $[1_G]$ is the identity $1=1_{\C_{par}G}$ of the algebra. Also, $\C_{par}G$ satisfies the following universal property (cf.\ \cite[page 512]{Dokuchaev-Exel-Piccione}): for any partial representation $\pi: G\to \End(V)$ of $G$, there exists a unique morphism of algebras $\psi_\pi:\C_{par}G \to \End(V)$ such that $\psi_\pi([g]) = \pi(g)$. To find the algebra that has the corresponding property for \(H\)-global \(G\)-partial representations, we have to impose some extra relations.

\begin{definition}
	The $\C$-algebra $\C^H_{par}G$ is generated by the symbols $\{[g]\mid g\in G\}$ subject to the relations
	\begin{equation} \label{eq:definingCparHG}
	\begin{aligned}
	[1_G][g]& =[g][1_G]=[g]\\
	[g^{-1}][g][g'] &  = [g^{-1}][gg']\\
	[g'][g][g^{-1}] & = [g'g][g^{-1}]\\
	[h][h^{-1}]& =[1_G]
	\end{aligned}
	\end{equation}
	for all $g,g'\in G$, $h\in H$. Equivalently, it is the quotient of $\C_{par}G$ by the ideal generated by $\{[h][h^{-1}]-[1_G]\mid h\in H\}$. 
\end{definition}

\begin{remark}\label{rem:parrepalg}
	Any $\C$-algebra $R$ can be realized as a subalgebra of an algebra of endomorphisms via
	\[
	\lambda\colon R \to \End_\C(R), \qquad a \mapsto \left[\lambda_a \colon b \mapsto ab\right].
	\]
	If we consider the obvious composition $G \to \C^H_{par}G \to \End_\C\left(\C^H_{par}G\right), g \mapsto \lambda_{[g]}$, then it follows from the defining relations of $\C^H_{par}G$ that this defines an $H$-global $G$-partial representation of $G$ in the $\C$-vector space $\C^H_{par}G$. 
\end{remark}

\begin{theorem}\label{thm:Iso}
	Given an $H$-global $G$-partial representation $\pi:G\to\End(V)$, there exists a unique morphism of $\C$-algebras $\phi_\pi\colon \C_{par}^HG \to \End(V)$ such that $\phi_\pi([g]) = \pi(g)$ for all $g\in G$. This makes $V$ into a left $\C_{par}^HG$-module. Conversely, given any left $\C_{par}^HG$-module $V$ with action $\mu\colon\C_{par}^HG \times V \to V$, the assignment
	\[
	\pi_\mu\colon G\to\End(V), \qquad g\mapsto \left[v\mapsto \mu([g],v)\right]
	\]
	is an $H$-global $G$-partial representation. These correspondences induce an isomorphism of categories
	\[
	\Phi \colon \GPRep{G}{H} \ \to \ {\C_{par}^HG}\text{-}\mathsf{Mod}
	\]
	between the category $\GPRep{G}{H}$ of $H$-global $G$-partial representations and the category ${\C_{par}^HG}\text{-}\mathsf{Mod}$ of $\C_{par}^HG$-modules.
\end{theorem}

\begin{proof}
	Let $(V,\pi)$ be an $H$-global $G$-partial representation. Consider the unique algebra morphism $\psi_\pi\colon \C_{par}G \to \End(V)$ such that $\psi_\pi([g])=\pi(g)$ for every $g\in G$. Clearly,
	\[
	\langle[h][h^{-1}]-[1_G]\mid h\in H\rangle\subseteq \ker(\psi_\pi),
	\]
	hence $\psi_\pi$ induces a well-defined algebra map $\phi_\pi:\C^H_{par}G \to \End(V)$ which is unique with the property that $\phi_\pi([g]) = \pi(g)$ for every $g\in G$. Conversely, assume that $V$ is a $\C_{par}^HG$-module with action $\mu\colon \C_{par}^HG \times V \to V$ and consider the map
	\[
	\pi_\mu\colon G\to\End(V), \qquad g\mapsto \left[v\mapsto \mu([g],v)\right].
	\]
	It satisfies $\pi_\mu(1_G)(v) = \mu([1_G],v) = \mu(1,v) = v$ for all $v\in V$, hence $\pi_\mu(1_G) = \Id_V$. Furthermore,
	\begin{align*}
	\pi_\mu(g^{-1})\pi_\mu(g)\pi_\mu(g')(v) & = \mu([g^{-1}],(\mu[g],(\mu([g'],v)))) = \mu([g^{-1}][g][g'],v) \\
	& = \mu([g^{-1}][gg'],v) = \pi_\mu(g^{-1})\pi_\mu(gg')(v)
	\end{align*}
	for all $v\in V$ and $g,g' \in G$, thus
	$
	\pi_\mu(g^{-1})\pi_\mu(g)\pi_\mu(g') = \pi_\mu(g^{-1})\pi_\mu(gg').
	$
	In an analogous way, one may check that
	$
	\pi_\mu(g')\pi_\mu(g)\pi_\mu(g^{-1}) = \pi_\mu(g'g)\pi_\mu(g^{-1}), 
	$
	so that $\pi$ is $G$-partial. Finally,
	\[
	\pi_\mu(h)\pi_\mu(h^{-1})(v) = \mu ([h][h^{-1}],v) = \mu(1,v) = v
	\]
	and hence $\pi_\mu$ is $H$-global. Now, assume that $f\colon V\to W$ is a morphism between the $H$-global $G$-partial representations $(V,\pi_V)$ and $(W,\pi_W)$. Recall that this means that
	\[
	f(\pi_V(g)(v)) = \pi_W(g)(f(v))
	\]
	for all $v\in V$ and $g\in G$. Using the notations above, we compute
	\[
	f(\phi_{\pi_V}([g])(v)) = f(\pi_V(g)(v)) = \pi_W(g)(f(v)) = \phi_{\pi_W}([g])(f(v))
	\]
	which implies that $f$ is $\C^H_{par}G$-linear. In fact, by definition of the action of $\C^H_{par}G$ on $V$, $f$ is $\C^H_{par}G$-linear if and only if $f\circ \pi_V(g) = \pi_W(g)\circ f$. Notice also that if $(V,\pi)$ is an $H$-global $G$-partial representation, then
	\[
	\begin{aligned}
	\pi_{\mu_{\pi}}(g)(v)  & = \mu_{\pi}([g],v) = \phi_\pi([g])(v) = \pi(g)(v), \\
	\mu_{\pi_{\mu}}([g],v) & = \phi_{\pi_\mu}([g])(v)=\pi_\mu(g)(v)=\mu([g],v),
	\end{aligned}
	\]
	for all $g\in G$ and $v\in V$, whence $\pi_{\mu_{\pi}} = \pi$, and $\mu_{\pi_{\mu}}=\mu$.
	Therefore, the assignments
	\[
	\xymatrix @R=0pt {
		\GPRep{G}{H} \ar@{<->}[r] & {\C_{par}^HG}\text{-}\mathsf{Mod} \\
		(V,\pi) \ar@{<->}[r] & (V,\mu_\pi) \\
		\left[f:(V,\pi_V)\to (W,\pi_W)\right] \ar@{<->}[r] & \left[f:(V,\mu_{\pi_V})\to (W,\mu_{\pi_W})\right]
	}
	\]
	provide well-defined functors between $\GPRep{G}{H}$ and ${\C_{par}^HG}\text{-}\mathsf{Mod}$, which form an isomorphism of categories.
\end{proof}

\begin{remark}
	Theorem \ref{thm:Iso} says, in particular, that notions like being isomorphic, being irreducible, being a direct sum, for objects in the two categories are in $1$-to-$1$ correspondence via the stated isomorphism. For instance, an object is irreducible in $\GPRep{G}{H}$ if and only if the corresponding object is irreducible in ${\C_{par}^HG}\text{-}\mathsf{Mod}$.
\end{remark}

\subsection{The groupoid algebra $\mathbb{C}\Gamma_H(G)$}\label{sec:groupoid}

It can be deduced from \cite[Theorem 2.6]{Dokuchaev-Exel-Piccione} that $G$-partial representations correspond to (usual) representations of the groupoid algebra of a certain groupoid $\Gamma(G)$. The idea is to extend this definition from \cite{Dokuchaev-Exel-Piccione} in order to handle $H$-global $G$-partial representations. The definitions and results of \cite{Dokuchaev-Exel-Piccione} can be recovered by setting $H=\{1_G\}$. 

In this section we will introduce a new groupoid \(\Gamma_H(G)\) naturally associated to \(G\) and \(H\) and show that the universal algebra constructed in Section~\ref{sec:ParRepMods} is isomorphic to \(\CC\Gamma_H(G)\), thus inducing an isomorphism of categories between the category of $H$-global $G$-partial representations and that of modules over this groupoid algebra.

\begin{definition}
	Let
	\[ \Gamma_H(G):=\{(A,g)\in \mathcal{P}_H(G/H)\times G\mid g^{-1}\in A \} \]
	be the groupoid with operation defined as
	\begin{equation} 
	(A,g)\cdot (B,h):= \begin{cases}
	(B, gh) & \text{if } A = hB, \\
	\text{not defined} & \text{otherwise.}
	\end{cases}
	\end{equation}
\end{definition}

This gives indeed a groupoid: the objects are the elements of $\mathcal{P}_H(G/H)$, the source and the range maps on $\Gamma_H(G)$ are $s(A,g)=A$ and $r(A,g)=gA$ and compisition is defined exactly when \(r(B, h) = s(A, g)\). It can easily be checked that 
the units are the elements \((A, 1_G)\) and that the inverse of \((A, g)\) is given by \((gA, g^{-1})\).

Hence we can look at the corresponding groupoid algebra $\mathbb{C}\Gamma_H(G)$: this algebra has $\Gamma_H(G)$ as basis over $\mathbb{C}$, and the product of two elements from $\Gamma_H(G)$ is set to be $0$ when it is not defined in $\Gamma_H(G)$.

The dimension of $\mathbb{C}\Gamma_H(G)$ over $\mathbb{C}$ is given by $|\Gamma_H(G)|$ which is easily computed as
\begin{equation}\label{eq:cardinality}
\begin{aligned} 
|\Gamma_H(G)| & =  2^{|G/H|-2}(|G|-|H|) + 2^{|G/H|-1}|H| \\
& =  2^{|G/H|-2}(|G|+|H|).
\end{aligned}
\end{equation}
See Example~\ref{ex:graphics2} for a detailed example.

\begin{remarks}
	Notice that this algebra becomes ``easier to handle'' (certainly smaller) when $|G/H|$ is small compared to $|G|$. This is one of our main motivations to develop our theory: we will put this into practice in Section~\ref{sec:SnSnminus1} when we will make explicit computations in the case $G=\mathfrak{S}_n$ and $H=\mathfrak{S}_{n-1}$.
\end{remarks}

In the following lemma we introduce a fundamental map from $G$ in the space $\C \Gamma_H (G)$: this will give us the link between the $H$-global $G$-partial representations and the representations of $\C \Gamma_H (G)$. 

\begin{lemma}\label{witnessinv}
	The map \(\mu_p : G \to \C \Gamma_H (G)\) defined by 
	\begin{equation}\label{s7eq:mu}
	\mu_p(g) \coloneqq \sum_{A\ni g^{-1}} (A, g)\quad \text{ for all }g\in G
	\end{equation}
	satisfies the following properties:
	\begin{align}
	\notag \mu_p(1_G)& =1_{\C\Gamma_H(G)}, \notag \\
	\mu_p(\bar{g}^{-1}) \mu_p(g) \mu_p(h) &  = \mu_p(\bar{g}^{-1}) \mu_p(gh), \label{s7eq:mupart} \\
	\notag \mu_p(h^{-1}) \mu_p(g^{-1}) \mu_p(\bar{g}) & = \mu_p(h^{-1} g^{-1}) \mu_p(\bar{g}) \notag
	\end{align}
	for any \(\bar{g}, g, h \in G\) such that \(g^{-1}\bar{g} \in H\). 
	In particular, the map $L_{\mu_p}\colon G\to \End(\C\Gamma_H(G))$ defined by setting \[L_{\mu_p}(g)(x)\coloneqq \mu_p(g)\cdot x\quad \text{ for all }g\in G,x\in \C\Gamma_H(G)\] gives an $H$-global $G$-partial representation.
\end{lemma}
\begin{proof}
	The first identity stated in \eqref{s7eq:mupart}, i.e.\ $\mu_p(1_G) = \sum_{A} (A, 1_G) = 1_{\C \Gamma_H(G)}$, is clear.
	
	Observe now that since
	\begin{equation}\label{eq:eqtech0}
	\left(\sum_{A\ni g^{-1}}(A,g)\right)(B,h) = \begin{cases} (B,gh) & \text{if }(gh)^{-1}\in B \\ \quad 0 & \text{otherwise}\end{cases}
	\end{equation}
	we have 
	\begin{equation}\label{eq:eqtech}
	\mu_p(g)\mu_p(h) = \left(\sum_{A\ni g^{-1}}(A,g)\right)\left(\sum_{B\ni h^{-1}}(B,h)\right) = \sum_{B \supseteq \{(gh)^{-1},h^{-1}\}}(B,gh).
	\end{equation}
	Assume that 
	\begin{equation}\label{eq:star1}
	g^{-1} \bar{g} \in H.
	\end{equation}
	Then, on the one hand
	\begin{align*} 
	\mu_p(\bar{g}^{-1}) \mu_p(g) \mu_p(h) &= \sum_{A \ni \bar{g} } (A, \bar{g}^{-1}) \sum_{B \ni g^{-1}} (B, g) \sum_{C \ni h^{-1}} (C, h)  \stackrel{\eqref{eq:eqtech}}{=} \sum_{A \ni \bar{g} } (A, \bar{g}^{-1}) \sum_{C \supseteq \{(gh)^{-1},h^{-1}\}}(C,gh)\\
	& \stackrel{\eqref{eq:eqtech}}{=} \sum_{C \supseteq \{(\bar{g}^{-1}gh)^{-1},(gh)^{-1},h^{-1}\}}(C,\bar{g}^{-1}gh) \stackrel{\eqref{eq:star1}}{=} \sum_{C \supseteq \{(gh)^{-1},h^{-1}\}} (C, \bar{g}^{-1} g h).
	\end{align*}
	On the other hand
	\begin{align*}
	\mu_p(\bar{g}^{-1}) \mu_p(gh) &= \sum_{A \ni \bar{g} } (A, \bar{g}^{-1}) \sum_{B \ni (gh)^{-1} } (B, gh) \stackrel{\eqref{eq:eqtech}}{=} \sum_{B\supseteq \{(\bar{g}^{-1}gh)^{-1},(gh)^{-1}\}} (B, \bar{g}^{-1} gh) \stackrel{\eqref{eq:star1}}{=} \sum_{B\supseteq \{h^{-1},(gh)^{-1}\}} (B, \bar{g}^{-1} gh),
	\end{align*}
	whence $\mu_p(\bar{g}^{-1}) \mu_p(g) \mu_p(h) = \mu_p(\bar{g}^{-1}) \mu_p(gh)$ and the second identity in \eqref{s7eq:mupart} is satisfied.
	The third identity in \eqref{s7eq:mupart} is proved in a similar way. The last statement follows immediately from Lemma~\ref{lem:HglobalCharacterization}.
\end{proof}

\begin{lemma}
	Let $H\subseteq G$ be a subgroup. The relations
	\begin{align} \label{s8eq:inverse}
	(A,1_G) & = \prod_{g\in A}\mu_p(g)\mu_p(g^{-1})\prod_{\bar{g}\in G\setminus A}(1_{\C \Gamma_H(G)} - \mu_p(\bar{g})\mu_p(\bar{g}^{-1})) \qquad \text{and} \\
	\notag (A,g')& = \mu_p(g')\prod_{g\in A}\mu_p(g)\mu_p(g^{-1})\prod_{\bar{g}\in G\setminus A}(1_{\C \Gamma_H(G)} - \mu_p(\bar{g})\mu_p(\bar{g}^{-1}))
	\end{align}
	hold in $\C \Gamma_H(G)$ for all $A\in\Pp_H(G/H)$ and all ${g'}^{-1}\in A$.
\end{lemma}
\begin{proof}
	All subsets of $G$ are tacitly assumed to be in $\Pp_H(G/H)$. In light of \eqref{eq:eqtech} we know that
	\begin{equation}\label{eq:eqtech2}
	\mu_p(g)\mu_p(g^{-1}) =\sum_{B\ni g} (B,1_G)
	\end{equation}
	for all $g\in G$. Recalling that
	\begin{equation}\label{eq:eqtech3}
	(A,1_G)(B,1_G) = \begin{cases} (B,1_G) & \text{if } A=B \\ \quad 0 & \text{otherwise}\end{cases}
	\end{equation}
	we compute
	\begin{equation}\label{eq:eqtech4}
	\begin{aligned}
	& \mu_p(g_1)\mu_p(g_1^{-1})\mu_p(g_2)\mu_p(g_2^{-1})\cdots \mu_p(g_t)\mu_p(g_t^{-1}) = \\
	& \stackrel{\eqref{eq:eqtech2}}{=}  \left(\sum_{B\ni g_1} (B,1_G) \right)\left(\sum_{B\ni g_2} (B,1_G) \right)\cdots \left(\sum_{B\ni g_t} (B,1_G) \right)\\
	& \stackrel{\eqref{eq:eqtech3}}{=} \sum_{B\supseteq \{g_1,\ldots,g_t\}} (B,1_G).
	\end{aligned}
	\end{equation}
	Summing up,
	\begin{equation}\label{eq:eqtech5}
	\begin{gathered}
	\prod_{g\in A}\mu_p(g)\mu_p(g^{-1}) \stackrel{\eqref{eq:eqtech4}}{=} \sum_{B\supseteq A} (B,1_G) \qquad \text{and} \\
	\prod_{\bar{g}\in G\setminus A}(1_{\C \Gamma_H(G)} - \mu_p(\bar{g})\mu_p(\bar{g}^{-1})) \stackrel{\eqref{eq:eqtech2}}{=} \prod_{\bar{g}\in G\setminus A}\left(\sum_{B\notni \bar{g}}(B,1_G)\right) \stackrel{\eqref{eq:eqtech3}}{=} \sum_{B\subseteq A}(B,1_G).
	\end{gathered}
	\end{equation}
	Thus,
	\begin{equation}\label{eq:eqtech6}
	\prod_{g\in A}\mu_p(g)\mu_p(g^{-1}) \prod_{\bar{g}\in G\setminus A}(1_{\C \Gamma_H(G)} - \mu_p(\bar{g})\mu_p(\bar{g}^{-1})) \stackrel{\eqref{eq:eqtech5}}{=} \left(\sum_{B\supseteq A} (B,1_G)\right)\left(\sum_{C\subseteq A}(C,1_G)\right) \stackrel{\eqref{eq:eqtech3}}{=} (A,1_G).
	\end{equation}
	Finally, if ${g'}^{-1}\in A$, then
	\[
	\mu_p(g')\prod_{g\in A}\mu_p(g)\mu_p(g^{-1})\prod_{\bar{g}\in G\setminus A}(1_{\C \Gamma_H(G)} - \mu_p(\bar{g})\mu_p(\bar{g}^{-1})) \stackrel{\eqref{eq:eqtech6}}{=} \left(\sum_{B\ni {g'}^{-1}}(B,g')\right)(A,1_G) \stackrel{\eqref{eq:eqtech0}}{=} (A,g').  \qedhere
	\]
\end{proof}

We are now ready to show the stated isomorphism. We will use the map $\mu_p\colon G \to \CC \Gamma_H(G)$ of Lemma~\ref{witnessinv}. 

\begin{theorem}\label{thm:iso}
	The map $\mu_p\colon G \to \CC \Gamma_H(G)$ induces an isomorphism of $\C$-algebras
	\[
	\begin{gathered}
	\xymatrix @R=0pt{
		\C^H_{par}G  \ar@{<->}[r] & \C\Gamma_H(G) \\
		[g] \ar@{|->}[r] & \displaystyle \sum_{A\ni g^{-1}} (A, g) \\
		[g]\cdot [P_A] & (A,g) \ar@{|->}[l]
	}
	\end{gathered}
	\]
	where $[P_A]\coloneqq \prod_{g\in A}[g][g^{-1}]\prod_{\bar{g}\in G\setminus A}(1-[\bar{g}][\bar{g}^{-1}])$.
\end{theorem}
\begin{proof}
	Consider the extension of the correspondence \eqref{s7eq:mu} to the free $\C$-algebra generated by the symbols $\{[g]\mid g\in G\}$, \ie
	\[
	\hat{\mu}_p([g]) \coloneqq \mu_p(g).
	\]
	It follows from the relations \eqref{s7eq:mupart} that $\hat{\mu}_p$ factors through the quotient defining $\C^H_{par}G$. As a consequence we have a well-defined $\C$-algebra morphism
	\[
	\mu\colon \C^H_{par}G \to \C\Gamma_H(G), \qquad [g] \mapsto \sum_{A\ni g^{-1}} (A, g).
	\]
	In the other direction, consider the assignment
	\[
	\mu^{-1}\colon \C\Gamma_H(G) \to \C^H_{par}G, \qquad (A,g) \mapsto [g]\cdot [P_A].
	\]
	A direct computation shows that, for ${g'}^{-1}\in A$,
	\[
	\mu(\mu^{-1}((A,g'))) = \mu([g']\cdot [P_A]) = \mu_p(g')\prod_{g\in A}\mu_p(g)\mu_p(g^{-1})\prod_{\bar{g}\in G\setminus A}(1-\mu_p(\bar{g})\mu_p(\bar{g}^{-1})) \stackrel{\eqref{s8eq:inverse}}{=} (A,g').
	\]
	To prove that also the other composition is the identity, observe that the elements $[P_A]$ in $\C_{par}^HG$ satisfy the same identities as the elements $P_A^\pi$ associated to an $H$-global $G$-partial representation $(V,\pi)$ that we saw in Section~\ref{sec:GeneralProperties}: indeed to prove those properties we only used the analogue of the defining relations \eqref{eq:definingCparHG}.
	
	So, combining Lemma~\ref{lem:sumP_AequalId} with Corollary~\ref{PAnonzero}, we have
	\[
	\mu^{-1}(\mu([g])) = \mu^{-1}\left(\sum_{A\ni g^{-1}} (A, g)\right) = \sum_{A\ni g^{-1}} [g]\cdot [P_A] \stackrel{\eqref{eq:gP_A=0}}{=} [g]\left(\sum_{A\in \Pp_H(G/H)} [P_A]\right) {=} [g]. \qedhere
	\]
\end{proof}

\begin{corollary}\label{cor:isocat}
	The \(H\)-global \(G\)-partial representations of \(G\) are in one-to-one correspondence with the (usual) representations of \(\C\Gamma_H(G)\). Namely, this correspondence is an isomorphism of categories. More precisely, given an algebra homomorphism \(\tilde{\pi}\colon \C\Gamma_H(G) \to \End(V)\), this determines an $H$-global $G$-partial representation $(V,\pi)$ with \(\pi \coloneqq \tilde{\pi} \circ \mu_p\); conversely, given an $H$-global $G$-partial representation $(V,\pi)$, there exists a unique algebra homomorphism \(\tilde{\pi}\colon \C\Gamma_H(G) \to \End(V)\) such that \(\pi = \tilde{\pi} \circ \mu_p\).
\end{corollary}

\begin{proof}
	It follows from Theorem~\ref{thm:Iso} and Theorem~\ref{thm:iso}.
\end{proof}

\begin{remark} 
	The argument of \cite{ParRep} can be adapted to show that the groupoid algebra \(\CC\Gamma_H(G)\) can be seen as a \emph{partial smash product} of a certain algebra \(C\) and the group algebra \(\CC G\). Indeed, write \(G = \coprod_{k = 1}^n g_k H\) such that \(g_1 = 1_G\) and let \(C\) be the commutative unital \(\CC\)-algebra generated by the \(n\) idempotents \(\varepsilon_k\), where \(\varepsilon_1 = 1_C\). Remark that \(\dim C = 2^n\). 
	Now, \(\CC G\) acts partially on \(C\) in the sense of \cite[Definition 3.4]{ParRep} via
	\[g \cdot \varepsilon_k = \varepsilon_l \varepsilon_m\]
	where \(g g_k = g_l h\) and \(g = g_m h'\) for \(h, h' \in H\). Remark that restriction to \(H\) gives a global action of \(H\) on \(C\).
	
	Consider on \(C \otimes \CC G\) the associative product
	\[(a \otimes g)(b \otimes \bar{g}) = a(g \cdot b) \otimes g\bar{g}\] The partial smash product \(\underline{C \# \CC G} = (C \otimes \CC G)(1_C \otimes 1_{\CC G})\) is generated by elements of the form 
	\[a \# g = a (g \cdot 1_C) \otimes g\]
	which satify \((a \# g)(b \# \bar{g}) = a(g \cdot b) \# g\bar{g}\) and \(a \# g = a (g \cdot 1_C) \# g\) (see \cite[Lemma 3.6]{ParRep}). If we set for \(A \in \mathcal{P}_H(G)\)
	\[P^{\#}_A \coloneqq \prod_{g_k \in A} \varepsilon_k \cdot\prod_{g_l \notin A} (1_C - \varepsilon_l),\]
	then it is easy to show that the map 
	\[\psi : \CC\Gamma_H(G) \to \underline{C\#\CC G} , (A, g) \mapsto P_{gA}^\# \# g\]
	is an isomorphism of \(\CC\)-algebras. 
\end{remark}

\subsection{Representation theory of $\mathbb{C}\Gamma_H(G)$} \label{sec:irreducibles}

In this section we describe the representation theory of the algebra $\C\Gamma_H(G)$. Some of the results (including their proofs) are natural extensions of results in \cite{Dokuchaev-Exel-Piccione}. They can also be deduced from the general theory developed in \cite{Steinberg-Moebius1,Steinberg-Moebius2} using the semigroup in Section~\ref{sec:semigroup} (cf.\ also \cite{Steinberg-Book}). 
The main point of our discussion is to make explicit the general constructions in our specific situation.

\medskip 

In \cite[Theorem 3.2]{Dokuchaev-Exel-Piccione} it was shown that \(\CC \Gamma_{\{1_G\}}(G)\) is a direct product of matrix algebras over the group algebras of the subgroups of $G$, hence \(\CC \Gamma_{\{1_G\}}(G)\) is a semisimple algebra. Here we use the same arguments to show that \(\CC \Gamma_H(G)\) is a semisimple algebra; this time the direct product runs only over certain subgroups.

\begin{definition}\label{def:trivial}
	Let \(K\) be a finite group and \(m \in \NN\). Denote by
	\[\Gamma_m^K = \{(k, i, j) \mid k \in K; i, j \in \{1, \dots, m\}\}\]
	the \emph{trivial groupoid} on the set $\{1,\ldots,m\}$ with group $K$ (see, for instance, \cite[Example 1.4]{Mackenzie} for the terminology). Source and range maps are given by
	\begin{align*}
	s(k,i, j) &= j \\
	r(k, i, j) &= i
	\end{align*}
	and the composition law by
	\[(k, i, j) \cdot (k', i', j') = \begin{cases}
	(kk', i, j') & \text{if } j = i', \\
	\text{not defined} & \text{if } j \neq i'.
	\end{cases}\]
\end{definition}

Given a groupoid \(\Gamma\), consider the oriented graph \(E_\Gamma\) whose vertices are the objects (units) of \(\Gamma\) and where there is an oriented edge from \(s(\gamma)\) to \(r(\gamma)\) for each \(\gamma \in \Gamma\). The following  proposition is  proved in \cite[Proposition 3.1]{Dokuchaev-Exel-Piccione} (see also \cite[Theorem 3.2]{Steinberg-Moebius1}).

\begin{proposition} \label{prop:connectedcomp}
	Let \(\Gamma\) be a groupoid such that \(E_\Gamma\) is connected and has a finite number of vertices \(m\). Let \(x\) be a vertex of \(E_\Gamma\) and \(K\) the isotropy group of \(x\), i.e.
	\[K = \{\gamma \in \Gamma \mid s(\gamma) = r(\gamma) = x\}.\]
	Then \(\Gamma \cong \Gamma_m^K\) and \(\CC\Gamma \cong M_m(\CC [K])\).
\end{proposition}

If \(E_\Gamma\) is not connected, then every connected component corresponds to a subgroupoid of \(\Gamma\). So if the number of vertices is finite (which is the case for \(\Gamma_H(G)\) when \(G/H\) is finite), then we can write \(\Gamma = \coprod_{i} \Gamma_i\) where \(\Gamma_i\) are subgroupoids with $E_{\Gamma_i}$ connected, and 
\begin{equation}\label{eq:semisimple}
\CC \Gamma = \prod_i \CC \Gamma_i \cong \prod_i M_{m_i}(\CC [K_i]) \cong \prod_{i,j} M_{m_i\cdot n_j}(\CC)
\end{equation}
for some groups \(K_i\), where $m_i=|\Gamma_i|$ and $\CC [K_i] \cong \prod_jM_{n_j}(\CC)$ is the Artin-Wedderburn decomposition of $\CC [K_i]$. In particular, \(\CC \Gamma\) is semisimple.

Now we want to better understand the case of the groupoid $\Gamma_H(G)$, in particular which groups $K_i$ appear and how the numbers $m_i$ are related to them.

\begin{remark}\label{rem:components}
	Let us pick an object $A_i\in \mathcal{P}_H(G/H)$ in a connected component $\Gamma_i$ of the groupoid $\Gamma_H(G)$. Any other object in $\Gamma_i$ is the range $r(A_i,g^{-1})=g^{-1}A_i$ of an element $(A_i,g^{-1})$ of $\Gamma_H(G)$, \ie for some $g\in A_i$.
	
	Let $K_i=K_{A_i}\coloneqq \{g\in G\mid gA_i=A_i \}$ be the stabilizer of the set $A_i$. Observe that $K_i$ is the isotropy group of $A_i$. Hence, by Proposition~\ref{prop:connectedcomp}, $\C\Gamma_i\cong M_{m_i}(\C[K_i])$, where $m_i=m_{A_i}$ is the number of objects in the connected component $\Gamma_i$. Let us show that $m_i=|A_i|/|K_i|$. Since $1_G\in A_i$ we have $K_i\subseteq A_i$ and $A_i$ is a disjoint union of right cosets of $K_i$, say $A_i=\coprod_{j=1}^mK_it_j$. We claim that the set of distinct objects of $\Gamma_i$ is $\{t_j^{-1}A_i\mid j=1,2,\dots,m \}$. Indeed, if $g\in A_i$, then $g=kt_j$ for some $j$ and some $k\in K_i$ and so $g^{-1}A_i=t_j^{-1}k^{-1}A_i=t_j^{-1}A_i$ is one of our objects. Moreover, $t_j^{-1}A_i=t_\ell^{-1}A_i$ implies $t_\ell t_j^{-1}\in K$, so that $Kt_j=Kt_\ell$ and hence $j=\ell$. Therefore, our objects are distinct and $m=m_i=m_{A_i}$, in fact.
\end{remark}

\begin{theorem} \label{thm:algebra_decomposition}
	Let $\Gamma_H(G)=\coprod_{i=1}^{[G:H]} \coprod_{j=1}^{c_i}\Gamma_{i,j}$ be the decomposition into connected components of $\Gamma_H(G)$, where the objects $A_{i,j}$ in $\Gamma_{i,j}$ have cardinality $i\cdot |H|$ for all $j=1,2,\dots,c_i$. Let $A_{i,j}$ be an object of $\Gamma_{i,j}$, $K_{i,j}\coloneqq K_{A_{i,j}}$ and $m_{i,j}\coloneqq m_{A_{i,j}}$ for all $i$ and $j$. Then we have the following algebra isomorphism 
	\begin{equation}\label{eq:AWdec}
	\C\Gamma_H(G)\cong \prod_{i=1}^{[G:H]} \prod_{j=1}^{c_i}M_{m_{i,j}}(\C[K_{i,j}]) \cong \prod_{i=1}^{[G:H]} \prod_{j=1}^{c_i} \prod_{k=1}^{e_{i,j}}M_{m_{i,j}\cdot d_k}(\CC),
	\end{equation}
	where $\CC [K_{i,j}] \cong \prod_{k=1}^{e_{i,j}}M_{d_k}(\CC)$ is the Artin-Wedderburn decomposition of $\CC [K_{i,j}]$. In particular, $\C\Gamma_H(G)$ is a semisimple ring.
\end{theorem}

\begin{example}
	Let $H$ be a subgroup of a finite group $G$ such that $[G:H]=2$. In this case there are only two objects in $\Gamma_H(G)$, of different cardinalities: $H$ and $G$, so Theorem~\ref{thm:algebra_decomposition} and \eqref{eq:cardinality} give $\C\Gamma_H(G)\cong \C[H]\times \C[G]$. 
\end{example}
See Example~\ref{ex:graphics2} for another computed example.

\smallskip

In order to understand the representation theory of $\C\Gamma_H(G)$, we outline how the representation theory of a finite dimensional semisimple associative unital algebra $A$ gets recovered from the representation theory of the algebras $eAe$ for the idempotents $e\in A$.

\medskip 

Let $A$ be a finite-dimensional associative semisimple unital algebra over $\mathbb{C}$, with $A\not\cong \mathbb{C}$. So, $A$ is the direct sum of matrix algebras by Wedderburn theory. Let $e\in A$ be a nontrivial idempotent of $A$, i.e.\ $0\neq e\neq 1$ (such an $e$ does exist since $A\not\cong \mathbb{C}$).

Given an $eAe$-module $V$ we define the $A$-module $\mathsf{Ind}_eV$ by setting
\[ \mathsf{Ind}_eV\coloneqq Ae\otimes_{eAe}V.  \]

The following theorem follows from standard theory of algebras. A proof is sketched in Appendix \ref{appendix:A} of the present article.

\begin{theorem} \label{thm:reps_eAe}
	If $W$ is an irreducible $eAe$-module, then $\mathsf{Ind}_eW$ is an irreducible $A$-module. Every irreducible $A$-module $V$ is isomorphic to $\mathsf{Ind}_eW$ for some nontrivial idempotent $e\in A$ and some irreducible $eAe$-module $W$.
\end{theorem}

We apply this construction to our algebra $\mathbb{C}\Gamma_H(G)$. In this case we have the idempotents $(A,1_G)$, with $A\in\mathcal{P}_H(G/H)$. Notice that
\[ (A,1_G) \mathbb{C}\Gamma_H(G) (A,1_G) = \mathrm{span}_\mathbb{C}\{(A,g)\mid g A=A\}. \]
So if we set $K=K_A\coloneqq \{g\in G\mid gA=A \}$ (notice that $A\supseteq K$), then 
\[ (A,1_G) \mathbb{C}\Gamma_H(G) (A,1_G) \cong \mathbb{C}[K].  \]

Now given an irreducible representation $(W,\rho)$ of $K$, we want to understand the $\C\Gamma_H(G)$-module
\begin{equation}\label{eq:induction}
\mathsf{Ind}_{A}W\coloneqq \mathsf{Ind}_{(A,1_G)}W= \mathbb{C}\Gamma_H(G)(A,1_G) \otimes_{\mathbb{C}[K]}W.
\end{equation}

Observe that 
\begin{align*}
\mathbb{C}\Gamma_H(G) (A,1_G) & = \mathrm{span}_\mathbb{C}\{(A,g)\mid g^{-1}H\subseteq A\} \\
& =  \mathrm{span}_\mathbb{C}\{(A,g)\mid g^{-1}\in  A\}\\
& =  \mathrm{span}_\mathbb{C}\{(A,g)\mid g\in  A^{-1}\},
\end{align*} 
so that
\[\mathsf{Ind}_{A}W= \mathbb{C}\Gamma_H(G)(A,1_G) \otimes_{\mathbb{C}[K]}W=\mathrm{span}_\mathbb{C}\{(A,g)\mid g\in  A^{-1}\}\otimes_{\mathbb{C}[K]}W. \]

Notice that $A^{-1}$ is left $H$-invariant and right $K$-invariant. In particular, in light of Remark \ref{rem:components}, we have $A^{-1}=\coprod_{i=1}^{m_A} t_i^{-1}K$ and therefore
\begin{equation}\label{eq:indA}
\mathsf{Ind}_{A}W= \mathrm{span}_\mathbb{C}\{(A,g)\mid g\in  A^{-1}\}\otimes_{\mathbb{C}[K]}W =\bigoplus_{i=1}^{m_A}\C \left(A,t_i^{-1}\right)\otimes_{\mathbb{C}[K]}W 
\end{equation}
as vector spaces. This suggests a better description of $\Ind_AW$ as $\mathbb{C}\Gamma_H(G)$-module: we can think of it as a restriction (analogous to the one in Remark~\ref{rem:ourconstr_restriction}) of the $G$-global representation $\mathsf{Ind}_K^GW$ to the subspace 
\[ \Ind_AW=\bigoplus_{t_i^{-1}K\in A^{-1}/K}W^{t_i^{-1}K} \subseteq \mathsf{Ind}_K^GW, \]
where $W^{t_i^{-1}K} = \C \left(A,t_i^{-1}\right)\otimes_{\mathbb{C}[K]}W \cong W$ as vector spaces and the inclusion $\varphi$ is uniquely determined by
\begin{equation}\label{eq:varphi}
\bigoplus_{i=1}^{m_A}\C \left(A,t_i^{-1}\right)\otimes_{\mathbb{C}[K]}W  \to \C[G]\otimes_{\C[K]} W, \qquad (A,t_i^{-1}) \otimes_{\mathbb{C}[K]} w \mapsto t_i^{-1} \otimes_{\mathbb{C}[K]} w.
\end{equation}
This naturally gives an $H$-global $G$-partial representation, as $A^{-1}$ is left $H$-invariant.

\begin{remark} \label{rem:isotroisomo}
	Let us keep the notation from above. We already observed in Remark \ref{rem:components} that the other elements of the connected component of $A$ are the $t_j^{-1}A$. It is clear that $t_j^{-1}Kt_j=\{g\in G\mid gt_j^{-1}A=t_j^{-1}A \}$ is the stabilizer of $t_j^{-1}A$, which is therefore isomorphic to the stabilizer $K$ of $A$. If we call $\rho_j\colon t_j^{-1}Kt_j\to \End(W_j)$, for $W_j\equiv W$, the map defined as $\rho_j(x)\coloneqq \rho(t_jxt_j^{-1})$ for all $x\in t_j^{-1}Kt_j$, then $(W_j,\rho_j)$ is clearly an irreducible representation of $t_j^{-1}Kt_j$, and it is easy to show that $\Ind_{A}W$ and $\Ind_{t_j^{-1}A}W_j$ are isomorphic $H$-global $G$-partial representations. Therefore, up to isomorphism, the irreducible $H$-global $G$-partial representation $\Ind_{A}W$ depends only on the connected component of $\Gamma_H(G)$ containing $(A,1_G)$ and not on the particular vertex chosen.
\end{remark}

\begin{theorem} \label{thm:irredsHglGpar}
	Let $\Gamma_H(G)=\coprod_{i=1}^{[G:H]} \coprod_{j=1}^{c_i}\Gamma_{i,j}$ be the decomposition into connected components of $\Gamma_H(G)$, where the objects in $\Gamma_{i,j}$ have cardinality $i\cdot |H|$ for all $j=1,2,\dots,c_i$. Fix an object $A_{i,j}$ of $\Gamma_{i,j}$ for all $i$ and $j$ and set $K_{i,j}\coloneqq K_{A_{i,j}}$ and $m_{i,j}\coloneqq m_{A_{i,j}}$. Then the $\Ind_{A_{i,j}}W$'s, as $W$ runs over all the inequivalent irreducible representations of $K_{i,j}$, are all the inequivalent irreducible $H$-global $G$-partial representations up to isomorphism.
\end{theorem}

\begin{proof}
	From Corollary \ref{cor:isocat} and Theorem~\ref{thm:reps_eAe} we know that the $\Ind_{A_{i,j}}W$'s are irreducible $H$-global $G$-partial representations. In addition, by resorting to Remark~\ref{rem:isotroisomo} as well, we conclude that any irreducible $H$-global $G$-partial representation is isomorphic to one of them. We are left to check that they are inequivalent. 
	
	By \eqref{eq:indA} we have $\dim_\C\left(\Ind_{A_{i,j}}W\right) = m_{i,j}\dim_\C(W)$, whence by Artin-Wedderburn theory
	\[
	\sum_{W}\dim_\C\left(\Ind_{A_{i,j}}W\right)^2 = \sum_{W}\left(m_{i,j}\dim_\C(W)\right)^2 = m_{i,j}^2\sum_{W}\dim_\C(W)^2 = m_{i,j}^2|K_{i,j}|
	\]
	where all the sums are over all the inequivalent irreducible representations $W$ of $K_{i,j}$.
	
	Summing over $i$ and $j$, and comparing with \eqref{eq:AWdec}, this shows that we cannot have redundancy among the irreducible representations that we have found.
\end{proof}

\begin{remark} \label{rem:lev1trivrep}
	It is noteworthy that the case \(i = 1\) in Theorem \ref{thm:irredsHglGpar} corresponds to \(H\)-global \(G\)-partial representations where any \(g \in G \setminus H\) acts as 0 (cf. Example \ref{ex:trivialHglGpar}), while the case \(i = [G : H]\) corresponds to the \(G\)-global representations. 
\end{remark}

\begin{example}
	Let $H$ be a subgroup of a finite group $G$ such that $[G:H]=2$. We already saw that $\C\Gamma_H(G)\cong \C[H]\times \C[G]$. The factor $\C[G]$	corresponds to the irreducible $G$-global representations, while the summand $\C[H]$ corresponds to the irreducible $H$-global representations where the elements of $G\setminus H$ act as $0$ (cf.\ Example~\ref{ex:trivialHglGpar}).  
	Both cases come from our construction (cf.\ Remark~\ref{rem:lev1trivrep}), so, by Theorem~\ref{thm:irredsHglGpar}, we found all the irreducible $H$-global $G$-partial representations in this case. 
\end{example}

\section{Representation theory: restriction, globalization and induction}

In this section we discuss some important constructions, like the restriction to $H$ and the globalization of an irreducible $H$-global $G$-partial representation, and a partial induction of a global representation of $H$ to an $H$-global $G$-partial representation.

\subsection{Restriction to $H$ of irreducibles} \label{sec:restrictionH}
 
Let $H$ be a subgroup of a finite group $G$, and let $(V,\pi)$ be an $H$-global $G$-partial representation. By definition, the restriction of $\pi$ to $H$ gives a global representation of $H$, denoted $\Res_H^GV$.  
In this section we describe this restriction for the irreducible $H$-global $G$-partial representations that we constructed in Section~\ref{sec:irreducibles}.

Recall that to get the $H$-global $G$-partial irreducibles, we started with an $A\in \mathcal{P}_H(G/H)$ and we considered an irreducible representation $(W,\rho)$ of the subgroup $K=K_A\coloneqq \{g\in G\mid gA=A \}$ of $G$. Then the corresponding irreducible $H$-global $G$-partial representation was given by $\Ind_{A}W$, i.e.\ by the restriction of the $G$-global representation $\mathsf{Ind}_K^GW$ to the subspace $\bigoplus_{t_i^{-1}K\in A^{-1}/K}W^{t_i^{-1}K} \subseteq \mathsf{Ind}_K^GW$ via the obvious inclusion and projection maps.   
We want to understand the restriction of $\Ind_{A}W$ to $H$.  
The answer to this problem is given by a well-known formula of Mackey suitably adapted to our ``restricted'' situation. We just need some more notation.

\begin{notation}[{cf.\ \cite[\S7.3]{SerreBook}}]\label{notation}
	Choose a set $S$ of representatives of $(H,K)$-double cosets of $A^{-1}$, i.e.\ $A^{-1}=\coprod_{s\in S}HsK$. For $s\in S$, let $K_s\coloneqq sKs^{-1}\cap H$, which is a subgroup of $H$. If we set
	\[ \rho^s(x)=\rho(s^{-1}xs),\quad \text{ for }x\in K_s, \]
	we obtain a representation $\rho^s\colon K_s\to \GL(W)$ of $K_s$, denoted $W_s$. Since $K_s$ is a subgroup of $H$, we can consider the induced representation $\Ind_{K_s}^HW_s$.
\end{notation}
The proof of the following theorem is identical to the one of \cite[Proposition 22]{SerreBook} (corresponding to the case $A=G$), so it will be omitted.
\begin{theorem} \label{thm:Hrestriction}
	The representation $\Res_{H}^{G}(\Ind_{A}W)$ of $H$ is isomorphic to the direct sum of the representations $\Ind_{K_s}^HW_s$, for $s\in S$.
\end{theorem}

\begin{example} \label{ex:Hrestriction}
	In the notation above, suppose that $A=KH$, so that $A^{-1}=HK=H1_GK$ (see Section~\ref{sec:SnSnminus1} for a concrete example). Then $S=\{1_G\}$, $K_{1_G}=K\cap H$ and $\rho^{1_G}$ is just the restriction of $\rho$ to $K\cap H$, so that
	\[ \Res_{H}^{G}(\Ind_{KH}W)\cong \Ind_{K\cap H}^H(\Res_{K\cap H}^KW).  \]
\end{example}

\subsection{Globalization of irreducibles} \label{sec:gloablirreps}

We proved in Theorem~\ref{thm:globalization} that every $G$-partial representation admits a globalization, unique up to isomorphism. In this section we will give an explicit description of the globalization of the irreducible $H$-global $G$-partial representations.

Let $A\in \mathcal{P}_H(G/H)$ and $(W,\rho)$ be an irreducible representation of the subgroup $K=K_A\coloneqq \{g\in G\mid gA=A \}$ of $G$. The corresponding irreducible representation is given by
\[\Ind_{A}W=\bigoplus_{t_i^{-1}K\in A^{-1}/K}\C t^{-1}_i \otimes_{\C[K]}W \subseteq \C[G] \otimes_{\C[K]} W \cong \bigoplus_{g_iK \in G/K} W^{g_iK}.\]

\begin{theorem} \label{thm:irredglobalization}
	The globalization of the irreducible $H$-global $G$-partial representation $\Ind_{A}W$ is given by $\mathsf{Ind}_K^GW$ with the obvious inclusion and projection maps.
\end{theorem}

\begin{proof}
	Set $U\coloneqq \mathsf{Ind}_K^GW$ and $V\coloneqq\Ind_{A}W$, for the sake of simplicity. Obviously $(U,\rho)$ is a global representation of $G$ and we already observed (cf.\ the discussion preceding Remark~\ref{rem:isotroisomo}) that $(V,\pi)$ is the restriction of $(U,\rho)$ via the inclusion $\varphi$ and projection $\tau$ as in Remark \ref{rem:ourconstr_restriction}. This proves properties \ref{item:GR1} and \ref{item:GR2} of a globalization. 
	
	Instead of proving property \ref{item:GR3}, we will check the properties \ref{item:GR3'} and \ref{item:GR4'} of Remark~\ref{rem:primeaxioms}. Property \ref{item:GR3'}, i.e.\ $U=\sum_{g \in G}\rho(g)(V)$, is obvious. In order to check property \ref{item:GR4'}, let $(U',\rho',\varphi',\tau')$ be another quadruple satisfying \ref{item:GR1} and \ref{item:GR2}. Consider the composition
	\[
	\begin{gathered}
	\xymatrix @R=0pt{
		W \ar[r] & \Ind_AW \ar[r]^{\varphi'} & U' \\ 
		w \ar@{|->}[r] & (A,1_G) \otimes_{\C[K]} w \ar@{|->}[r] & \varphi'\left((A,1_G) \otimes_{\C[K]} w\right)
	}
	\end{gathered}.
	\]
	It satisfies
	\[
	\begin{gathered}
	\varphi'\left((A,1_G) \otimes_{\C[K]} \rho(k)(w)\right) = \varphi'\left((A,k) \otimes_{\C[K]} w\right) = \varphi'\left(\left(\sum_{B\ni k^{-1}}(B,k)\right)(A,1_G) \otimes_{\C[K]} w\right) \\
	= \varphi'\left(\pi(k)\Big((A,1_G) \otimes_{\C[K]} w\Big)\right) \stackrel{\ref{item:RR2}}{=} \rho'(k)\left(\varphi'\Big((A,1_G) \otimes_{\C[K]} w\Big)\right),
	\end{gathered}
	\]
	which means that it is a $K$-homomorphism and so there exists a unique $G$-homomorphism $\psi\colon U\to U'$ such that
	\begin{equation}\label{eq:psivarphi'}
	\psi(1_G \otimes_{\C[K]} w) = \varphi'\left((A,1_G) \otimes_{\C[K]} w\right),
	\end{equation}
	by the universal property of $\mathsf{Ind}_K^G(W)$. In turn, $\psi$ satisfies
	\[
	\begin{gathered}
	\psi\left(\varphi((A,t_i^{-1}) \otimes_{\C[K]} w)\right) \stackrel{\eqref{eq:varphi}}{=} \psi\left(t_i^{-1} \otimes_{\C[K]} w\right) = \rho'(t_i^{-1})\psi\left(1_G \otimes_{\C[K]} w\right) \\
	\stackrel{\eqref{eq:psivarphi'}}{=} \rho'(t_i^{-1})\varphi'\left((A,1_G) \otimes_{\C[K]} w\right) \stackrel{\eqref{eq:varphipi}}{=} \varphi'\left((A,t_i^{-1}) \otimes_{\C[K]} w\right)
	\end{gathered}
	\]
	Being already a $G$-homomorphism, $\psi$ is the map required in \ref{item:GR4'}, completing the proof.
\end{proof}

\subsection{Induction from $H$-global to $H$-global $G$-partial}\label{sec:induction}

Consider a subgroup $H$ of a finite group $G$. Given a $G$-partial representation $(V,\pi)$, it is clear that the restriction $\Res_H^G(\pi)\coloneqq \pi_{|_H}\colon H\to \End(V)$ gives an $H$-partial representation, that we denote $\Res_H^GV$. By definition, $(V,\pi)$ is $H$-global if and only if $\Res_H^GV$ is a global representation of $H$.

For global representations, there is the well-known converse construction of the latter: the induction $\Ind_H^GW$ of a global representation $W$ of $H$, that we already used in this paper. This representation satisfies a universal property, i.e.\ it is equipped with an $H$-homomorphism $\eta_W\colon W\to \Ind_H^GW$ such that for any $H$-homomorphism $f\colon W\to \Res_H^GU$ into a $G$-global representation $U$, there exists a unique $G$-homomorphism $\tilde{f}\colon \Ind_H^GW\to U$ such that $f=\tilde{f}\circ \eta_W$, i.e.\ the following diagram commutes
\[
\begin{gathered}
\xymatrix{
	\Ind_H^GW \ar@{.>}[dr]^-{\exists !\tilde{f}} & \\
	W\ar[u]^-{\eta_W} \ar[r]_-{f} & \Res_H^GU\equiv U.
} \end{gathered}
\]
Equivalently, there is a bijection
\[
\begin{gathered}
\xymatrix @R=0pt {
	\Hom_{G}\left(\Ind_H^GW,U\right) \ar@{<->}[r]^{\cong} & \Hom_{H}\left(W,\Res^G_HU\right) \\
	f' \ar@{|->}[r] & f'\circ \eta_W \\
	\tilde{f} & f \ar@{|->}[l]
}
\end{gathered}
\]
which is known as \emph{Frobenius reciprocity}.

It is now natural to ask if starting with a global representation $W$ of $H$, we can construct a \emph{partial induction} which is an $H$-global $G$-partial representation satisfying a similar reciprocity.

We propose the following definition.
\begin{definition}
	The \emph{partial induction} of a global representation $W$ of $H\subseteq G$ to $G$ is an $H$-global $G$-partial representation $\PInd_H^GW$ equipped with an $H$-homomorphism $\eta_W\colon W\to \PInd_H^GW$, such that for every $H$-homomorphism $f\colon W \to \Res_H^GU\equiv U$ from $W$ to an $H$-global $G$-partial representation $U$, there exists a unique morphism of $G$-partial representations $\tilde{f}\colon \PInd_H^GW \to U$ such that $\tilde{f} \circ \eta_W = f$.
\end{definition}

In the present section we prove the existence of such a partial induced representation by providing an explicit construction. Notice that, in light of the universal property that defines it, if an induced partial representation exists, then it is necessarily unique. Therefore, we will refer to it as \emph{the} partial induced representation.

Consider a subgroup $H$ of a finite group $G$, and let $G=\coprod_{i=1}^rg_iH$ with $g_1=1_G$. Given an $H$-global representation $(W,\rho)$, we define the vector space
\[ \PPInd _H^GW \coloneqq \bigoplus_{A\in\mathcal{P}_H(G/H)} \bigoplus_{g_iH \in A/H}W^{A,i}  \]
where the $W^{A,i}$ are vector spaces equipped with linear isomorphisms $\phi_{A,i}\colon W\to W^{A,i}$.

We define $\widetilde{\rho}\colon G\to \mathsf{End}(\PPInd _H^GW)$ by setting for all $g\in G$, $w\in W$, $A\in \mathcal{P}_H(G/H)$ and $g_iH\subseteq A$
\begin{equation}\label{s8eq:defInd}
\widetilde{\rho}(g)(\phi_{A,i}(w)) \coloneqq \left\{ \begin{array}{ll}
0 & \text{ if }g^{-1}H\not\subseteq A\\
\phi_{gA,j}(\rho(h)(w)) & \text{ if }g^{-1}H \subseteq A\text{ and }gg_i=g_jh,\text{ with }h\in H
\end{array}\right.\, .
\end{equation}

\begin{proposition}\label{s12prop:PInd}
	The pair $(\PPInd _H^GW,\widetilde{\rho})$ is an $H$-global $G$-partial representation.
\end{proposition}

\begin{proof}
	The fact that $\widetilde{\rho}$ is $H$-global is clear from the definition. The fact that it is $G$-partial is a tedious but straightforward verification, that we leave to the reader.
\end{proof}

\begin{remark}
	For any $A \in \Pp_H\left( G/H\right)$, consider the orthogonal idempotent $P_A\coloneqq P_A^{\widetilde{\rho}}$ as in Definition \ref{PA}. Recall from Proposition \ref{s6prop:PA} that
	\[
	P_A=P_A^{\widetilde{\rho}} = \prod_{g_kH\subseteq A}\widetilde{\rho}(g_k)\widetilde{\rho}(g_k^{-1})\prod_{g_iH\subseteq G\setminus A}\left(\Id_{\PPInd_H^GW}-\widetilde{\rho}(g_i)\widetilde{\rho}(g_i^{-1})\right).
	\]
	It can be easily checked that $\bigoplus_{g_iH\subseteq A}W^{A,i} = (\PPInd_H^GW)^A=P_A(\PPInd_H^GW)$ (cf.\ Section~\ref{sec:GeneralProperties}).
\end{remark}

\begin{lemma}
	The function
	\begin{equation}\label{s8eq:eta}
	\eta_W\colon W \to \PPInd _H^GW, \qquad w\mapsto \sum_{A\in \mathcal{P}_H(G/H)}\phi_{A,1}(w)
	\end{equation}
	is an $H$-homorphism. Furthermore, it satisfies
	\begin{equation}\label{s8eq:etaPhi}
	\phi_{A,i}(w) = \widetilde{\rho}(g_i)P_{g_i^{-1}A}\eta_W\left(\phi_{g_i^{-1}A,1}^{-1}\widetilde{\rho}\left(g_i^{-1}\right)\left(\phi_{A,i}(w)\right)\right)
	\end{equation}
	for every $w\in W$, $A\in\mathcal{P}_H(G/H)$ and all $i=1,\ldots,r$ such that $g_iH\subseteq A$.
\end{lemma}

\begin{proof}
	Let $h\in H$. Then, using \eqref{s8eq:defInd}, we compute
	\begin{align*}
	\widetilde{\rho}(h)\left(\eta_W(w)\right) & = \widetilde{\rho}(h)\left(\sum_{A\in \mathcal{P}_H(G/H)}\phi_{A,1}(w)\right) = \sum_{A\in \mathcal{P}_H(G/H)}\widetilde{\rho}(h)\left(\phi_{A,1}(w)\right) \\
	& = \sum_{A\in \mathcal{P}_H(G/H)}\phi_{hA,1}\left(\rho(h)(w)\right) = \sum_{B\in \mathcal{P}_H(G/H)}\phi_{B,1}\left(\rho(h)(w)\right)\\
	& = \eta_W\left(\rho(h)(w)\right),
	\end{align*}
	which proves the $H$-linearity. To show \eqref{s8eq:etaPhi}, using again \eqref{s8eq:defInd}, we compute
	\begin{align*}
	\widetilde{\rho}(g_i)P_{g_i^{-1}A}\eta_W & \left(\phi_{g_i^{-1}A,1}^{-1}\widetilde{\rho}\left(g_i^{-1}\right)\left(\phi_{A,i}(w)\right)\right)  = \widetilde{\rho}(g_i)P_{g_i^{-1}A}\eta_W\left(\phi_{g_i^{-1}A,1}^{-1}\left(\phi_{g_i^{-1}A,1}(w)\right)\right) \\
	& = \widetilde{\rho}(g_i)P_{g_i^{-1}A}\eta_W\left(w\right) = \widetilde{\rho}(g_i)\left(\phi_{g_i^{-1}A,1}(w)\right) = \phi_{A,i}(w). \qedhere
	\end{align*}
\end{proof}

We want to show that $(\PPInd _H^GW,\widetilde{\rho})$ together with $\eta_W$ given by \eqref{s8eq:eta} satisfies the universal property of the partial induced representation.

\begin{remark}
	By \eqref{s8eq:defInd}, for every $h\in H$
	\[ \widetilde{\rho}(h) \circ \phi_{A,1} = \phi_{hA,1} \circ \rho(h) \]
	so that 
	\begin{equation}\label{s8eq:RhoPhi}
	\phi_{hA,1}^{-1} \circ \widetilde{\rho}(h)   =  \rho(h) \circ \phi_{A,1}^{-1}.
	\end{equation}
\end{remark}

Let $(U,\alpha)$ be an $H$-global $G$-partial representation and let $f\colon W\to \mathsf{Res}_H^GU\equiv U$ be an $H$-homomorphism. First of all, observe that if $F\colon \PPInd _H^GW \to U$ is any morphism of $G$-partial representations such that $F \circ \eta_W = f$, then for every $x=\phi_{A,i}(w)$, where $A\in\mathcal{P}_H\left(G/H\right)$ and $g_iH\subseteq A$, using \eqref{s8eq:etaPhi} we have
\begin{align*}
F\left(x\right) & = F\left(\widetilde{\rho}(g_i)P_{g_i^{-1}A}^{\widetilde{\rho}}\eta_W\left(\phi_{g_i^{-1}A,1}^{-1}\widetilde{\rho}\left(g_i^{-1}\right)\left(x\right)\right)\right)\\
&  = \alpha(g_i)P_{g_i^{-1}A}^{\alpha}\left(F\left(\eta_W\left(\phi_{g_i^{-1}A,1}^{-1}\widetilde{\rho}\left(g_i^{-1}\right)\left(x\right)\right)\right)\right) \\
& = \alpha(g_i)P_{g_i^{-1}A}^{\alpha}\left(f\left(\phi_{g_i^{-1}A,1}^{-1}\widetilde{\rho}\left(g_i^{-1}\right)\left(x\right)\right)\right).
\end{align*}
Therefore, we define $\hat{f}\colon \PPInd _H^GW\to U$ by setting, for $x=\phi_{A,i}(w)\in W^{A,i}$,
\[ \hat{f}(x)\coloneqq \alpha(g_i)P_{g_i^{-1}A}^\alpha f\left(\phi_{g_i^{-1}A,1}^{-1}\widetilde{\rho}(g_i^{-1})(x)\right).  \]

\begin{lemma}
	The map $\hat{f}$ is well defined, i.e.\ it does not depend on the chosen representatives $\{g_1,\ldots,g_r\}$. 
\end{lemma}

\begin{proof}
	Given $h\in H$, replacing $g_i$ by $g_ih$ we get
	\begin{align*}
	\alpha(g_ih)P_{h^{-1}g_i^{-1}A}^\alpha f\left(\phi_{h^{-1}g_i^{-1}A,1}^{-1}\widetilde{\rho}(h^{-1}g_i^{-1})(x)\right) & \stackrel{\phantom{(12.4)}}{=} \alpha(g_i)\alpha(h)P_{h^{-1}g_i^{-1}A}^\alpha f\left(\phi_{h^{-1}g_i^{-1}A,1}^{-1}\widetilde{\rho}(h^{-1}g_i^{-1})(x)\right)\\
	& \hspace{2pt}\stackrel{\eqref{s3eq:PAh}}{=}    \alpha(g_i)P_{g_i^{-1}A}^\alpha\alpha(h) f\left(\phi_{h^{-1}g_i^{-1}A,1}^{-1}\widetilde{\rho}(h^{-1}g_i^{-1})(x)\right)\\
	& \stackrel{\phantom{(12.4)}}{=}    \alpha(g_i)P_{g_i^{-1}A}^\alpha f\left(\rho(h)\phi_{h^{-1}g_i^{-1}A,1}^{-1}\widetilde{\rho}(h^{-1}g_i^{-1})(x)\right)\\
	& \stackrel{\eqref{s8eq:RhoPhi}}{=} \alpha(g_i)P_{g_i^{-1}A}^\alpha f\left(\phi_{g_i^{-1}A,1}^{-1}\widetilde{\rho}(h)\widetilde{\rho}(h^{-1}g_i^{-1})(x)\right)\\ 
	& \stackrel{\phantom{(12.4)}}{=}    \alpha(g_i)P_{g_i^{-1}A}^\alpha f\left(\phi_{g_i^{-1}A,1}^{-1}\widetilde{\rho}(g_i^{-1})(x)\right). \qedhere
	\end{align*}
\end{proof}

\begin{lemma}
	The map $\hat{f}$ is a morphism of $G$-partial representations. 
\end{lemma}

\begin{proof}
	Given $g\in G$, let $gg_i=g_jh$ with $h\in H$. Assume initially that $g^{-1}\in A$ and pick $0\neq x=\phi_{A,i}(w)$ arbitrarily. We have that 
	\begin{enumerate}[label=(\alph*),ref=(\alph*)]
		\item\label{s11item:1} $g_iH\subseteq A$, and $g_i^{-1}\in g_i^{-1}A$;
		\item\label{s11item:2} $g_j^{-1}gg_i=h\text{ and }g^{-1}g_jH=g_ih^{-1}H\subseteq A$;
		\item\label{s11item:3} $g_j^{-1}g_j=g_1\text{ and } g_jH \subseteq gA$;
		\item\label{s11item:4} $gg_i=g_jh\text{ and } g^{-1}H \subseteq A$
	\end{enumerate} 
	and therefore
	\begin{align*}
	\alpha(g)\hat{f}(x)& = \alpha(g_jhg_i^{-1})\alpha(g_i)P_{g_i^{-1}A}^\alpha f\left(\phi_{g_i^{-1}A,1}^{-1}\widetilde{\rho}(g_i^{-1})(x)\right) = \alpha(g_j)\alpha(h)\alpha(g_i^{-1})\alpha(g_i)P_{g_i^{-1}A}^\alpha f\left(\phi_{g_i^{-1}A,1}^{-1}\widetilde{\rho}(g_i^{-1})(x)\right)\\
	& \overset{\ref{s11item:1}}{=} \alpha(g_j)\alpha(h) P_{g_i^{-1}A}^\alpha f\left(\phi_{g_i^{-1}A,1}^{-1}\widetilde{\rho}(g_i^{-1})(x)\right) \overset{\eqref{s3eq:PAh}}{=} \alpha(g_j) P_{hg_i^{-1}A}^\alpha\alpha(h) f\left(\phi_{g_i^{-1}A,1}^{-1}\widetilde{\rho}(g_i^{-1})(x)\right)\\
	& = \alpha(g_j) P_{g_j^{-1}gA}^\alpha f\left(\rho(h)\phi_{g_i^{-1}A,1}^{-1}\widetilde{\rho}(g_i^{-1})(x)\right) \overset{\eqref{s8eq:RhoPhi}}{=} \alpha(g_j) P_{g_j^{-1}gA}^\alpha f\left( \phi_{hg_i^{-1}A,1}^{-1}\widetilde{\rho}(h)\widetilde{\rho}(g_i^{-1})(x)\right)\\
	& = \alpha(g_j) P_{g_j^{-1}gA}^\alpha f\left( \phi_{g_j^{-1}gA,1}^{-1}\widetilde{\rho}(g_j^{-1}g)(x)\right) = \alpha(g_j) P_{g_j^{-1}gA}^\alpha f\left( \phi_{g_i^{-1}gA,1}^{-1}\widetilde{\rho}(g_j^{-1}g)(\phi_{A,i}(w))\right)\\
	& \overset{\ref{s11item:2}}{=} \alpha(g_j) P_{g_j^{-1}gA}^\alpha f\left( \phi_{g_j^{-1}gA,1}^{-1} \phi_{g_j^{-1}gA,1}(\rho(h)(w))\right) \overset{\ref{s11item:3}}{=} \alpha(g_j) P_{g_j^{-1}gA}^\alpha f\left( \phi_{g_j^{-1}gA,1}^{-1} \widetilde{\rho}(g_j^{-1}) \phi_{gA,j}(\rho(h)(w))\right)\\
	& = \hat{f}(\phi_{gA,j}(\rho(h)(w))) \overset{\ref{s11item:4}}{=} \hat{f}(\widetilde{\rho}(g)(\phi_{A,i}(w))) = \hat{f}(\widetilde{\rho}(g)(x)).
	\end{align*}
	If, on the other hand, $g^{-1}\notin A$, then by definition $\widetilde{\rho}(g)(x)=0$, but also $\alpha(g_j) P_{g_j^{-1}gA}^\alpha =0$, since $g_j^{-1}\notin g_j^{-1}gA$. So $\alpha(g)\hat{f}(x)=0$ follows from the first few lines of the same computation.
\end{proof}

By construction, $\hat{f}$ is the unique morphism of $G$-partial representations such that $\hat{f}\circ \eta_W = f$. This proves that $(\PPInd_H^GW,\tilde{\rho})$ is the partial induction of $W$, as we wanted. Notice that in particular we established the following \emph{Frobenius reciprocity}:
\begin{equation}\label{s11eq:Frobenius}
\Hom_G(\PInd_H^GW,U) \cong  \Hom_H(W,\Res_H^GU).
\end{equation}

We conclude this section with the following remark, showing some advantages of working with the groupoid algebra $\C\Gamma_H(G)$ and its modules.

\begin{remark}
	Assume we are given a finite group $G$ and two subgroups $H,K$.
	\begin{enumerate}[label=(\alph*)]
		\item Using the restriction to $H$ of the map $\mu_p\colon G\to \C\Gamma_H(G)$ defined in \eqref{s7eq:mu}, we can see $\C\Gamma_H(G)$ as a right $\C[H]$-module (cf.\ Lemma~\ref{witnessinv}). Given a global representation $W$ of $H$, i.e.\ a left $\C[H]$-module, it turns out that $\PInd_H^GW \cong \C\Gamma_H(G)\tensor{\C [H]}W$ as left $\C\Gamma_H(G)$-modules, and hence as $H$-global $G$-partial representations. 
		
		\item More generally, given a $\left(\C\Gamma_H(G),\C [K]\right)$-bimodule $Q$ and a left $\C [K]$-module $W$, $Q\otimes_{\C [K]} W$ is a left $\C\Gamma_H(G)$-module and so an $H$-global $G$-partial representation. In the particular situation of Section~\ref{sec:genconstruction}, we consider a subset $A\subseteq G$ which is a union of $(H,K)$-double cosets (\eg $A = HK \subset G$) and the set $A/K$ of left cosets of $K$ contained in $A$. Then we can consider the partial action of $G$ on $A$ given by restriction of the global action of $G$ on $G$ by left multiplication. 
		By linearization as in Definition \ref{s1def:linearization}, $\C[A]$ becomes an $H$-global $G$-partial representation and so a left $\C\Gamma_H(G)$-module. It is also a right $\C [K]$-module and the two structures are compatible, whence it is a $\left(\C\Gamma_H(G),\C [K]\right)$-bimodule. Now,
		\[
		W_A \coloneqq \bigoplus_{g_iK\in A/K} W^{g_i} \cong \C [A]\otimes_{\C [K]} W,
		\]
		which is then an $H$-global $G$-partial representation. This allows us to recover the general construction performed in Section \ref{sec:genconstruction} and the induction construction \eqref{eq:induction} from \S\ref{sec:irreducibles}.
	\end{enumerate}
	We omit the details.
\end{remark}

In Section~\ref{sec:SnSnminus1} we will show some interesting computations of partial induced representations.

\section{The point of view of inverse semigroups}
\label{sec:semigroup}

In this section we define a semigroup $S_H(G)$ which is closely related to the semigroup $S(G)$ of Exel \cite{Exel-semigroup-PAMS} and that plays the same role for $H$-global $G$-partial representations as $S(G)$ does for $G$-partial representations.

\medskip

We can define a semigroup $S_H(G)\coloneqq \{[(A,g)]\mid (A,g)\in \Gamma_H(G)\}$ by setting
\[ [(A,g)]\cdot [(B,h)]\coloneqq  
[(h^{-1}A\cup B,gh)]  . \]
Notice that this operation is well defined as $h^{-1}g^{-1}\in h^{-1}A$.

It is easy to check that this gives indeed a semigroup (one just needs to check the associativity).
\begin{remark}
	Consider the monoid $S\coloneqq (\mathcal{P}(G/H),\cup)$, where $\mathcal{P}(G/H)\coloneqq \{A\subseteq G\mid Ah=A\text{ for all }h\in H\}$. Now $G$ has a right action on $\mathcal{P}(G/H)$ defined by $A\cdot g\coloneqq g^{-1}A$ for every $A\in \mathcal{P}(G/H)$ and $g\in G$. The \emph{semidirect product}  $S\ast  G$ is defined by endowing the set $S \times G$ with the composition law
	\[ (A,g)(B,h)  \coloneqq ((A\cdot h)\cup B,gh)=(h^{-1}A\cup B,gh) . \]
	Now $S\ast  G$ is a semigroup and clearly $S_H(G)$ is a subsemigroup of $S\ast  G$ (it is not a submonoid because $[(\varnothing,1_G)]$ is not in $S_H(G)$ as $H\nsubseteq \varnothing$).
	
	Observe also that $S_H(G)$ is a subsemigroup of $S_{\{1_G\}}(G)$ for any subgroup $H$ of $G$.
\end{remark}

\begin{remark}
	It is easy to see that for $H=\{1_G\}$ our semigroup $S_{\{1_G\}}(G)$ is anti-isomorphic to the semigroup $S(G)$ of Exel \cite{Exel-semigroup-PAMS} (cf. also \cite[Theorem 2.4]{KellenLawson}). 
\end{remark}

The semigroup $S_H(G)$ is in fact an inverse semigroup: given $[(A,g)]\in S_H(G)$, it is straightforward to check that its unique inverse is $[(gA,g^{-1})]$:
\[  [(A,g)]\cdot [(gA,g^{-1})]\cdot [(A,g)] = [(gA,1_G)] \cdot [(A,g)]=[(A,g)]  \]
and
\[  [(gA,g^{-1})]\cdot [(A,g)]\cdot [(gA,g^{-1})] = [(A,1_G)] \cdot [(gA,g^{-1})]=[(gA,g^{-1})].  \]
The idempotents of $S_H(G)$ are clearly the elements of the form $[(A,1_G)]$ with $A\in \mathcal{P}_H(G)$.

The natural partial order (cf.\ \cite[page 21]{Lawson}) of the inverse semigroup $S_H(G)$ is given by containment in the first component and equality in the second one: given $[(A,g)],[(B,h)]\in S_H(G)$, $[(B,h)]\leq [(A,g)]$ if and only if there exists $[(C,1_G)]$ such that 
\[ [(B,h)]=[(C,1_G)]\cdot [(A,g)]=[(g^{-1}C\cup A,g)] \]
so that $h=g$ and $B=g^{-1}C\cup A$, if and only if there exists $[(D,1_G)]$ such that 
\[ [(B,h)]= [(A,g)]\cdot [(D,1_G)] = [(D\cup A,g)] \]
so that $h=g$ and $B=D\cup A$.
These conditions are equivalent to $g=h$ and $B\supseteq A$.
\begin{remark}
	It is easy to check that the groupoid associated to the inverse semigroup $S_H(G)$ as in \cite[Proposition 4 of \S3.1]{Lawson} is exactly $\Gamma_H(G)$.
\end{remark}

By \cite[Theorem~4.2]{Steinberg-Moebius1}, we know that the map $f\colon \C S_H(G)\to \C\Gamma_H(G)$ defined by \[f([(A,g)])\coloneqq \sum_{(B,h)\leq (A,g)}(B,h)=\sum_{B\supseteq A}(B,g)\] is an isomorphism of algebras: to check that it is a homomorphism, we compute
\begin{align*}
f([(A,g)])\cdot f([(B,h)]) & = \sum_{C\supseteq A}(C,g) \sum_{D\supseteq B}(D,h) = \sum_{C\supseteq A}\sum_{D\supseteq B}(C,g) \cdot (D,h) \\
& = \mathop{\sum_{D\supseteq B}}_{hD\supseteq A}(D,gh) = \mathop{\sum_{D\supseteq B}}_{D\supseteq h^{-1}A}(D,gh) = \sum_{D\supseteq h^{-1}A\cup B}(D,gh)\\
& = f([(h^{-1}A\cup B,gh)]) = f([(A,g)]\cdot[(B,h)]).
\end{align*}
The inverse function of $f$ can be given explicitly by using M\"{o}bius inversion: cf.\ \cite{Steinberg-Moebius1}.

As we have seen in Theorem \ref{thm:iso}, the groupoid algebra $\C\Gamma_H(G)$ is isomorphic to the partial group algebra $\C^{H}_{par}G$. In view of the above discussion, we know that the groupoid algebra is also isomorphic to the semigroup algebra $\C S_H(G)$. By joining the two isomorphisms, we find out that $\C^{H}_{par}G \cong \C S_H(G)$ via the explicit isomorphism
\[
\C S_H(G) \to \C^{H}_{par}G, \qquad [(A,g)] \mapsto \sum_{B\supseteq A}[g][P_B].
\]
In view of Theorem \ref{thm:Iso}, this induces an isomorphism of categories between the category $\GPRep{G}{H}$ of $H$-global $G$-partial representations and the category $\C S_H(G)$-$\mathsf{Mod}$ of $\C S_H(G)$-modules. Though we will not do it here, it is now clear that all the results that we found in this work about $H$-global $G$-partial representations can be recast and understood in terms of the representation theory of the inverse semigroup $S_H(G)$ (cf.\ \cite{Steinberg-Moebius1,Steinberg-Moebius2}).

\begin{remark} \label{rem:S_H(G)_rep}
	Notice that $S_H(G)$ admits a representation in $M_{[G:H]}(\C[(\{\bar{0},\bar{1}\},\cdot)])$, where $(\{\bar{0},\bar{1}\},\cdot)$ is simply the multiplicative semigroup defined by $\bar{0}\cdot \bar{1}=\bar{0}\cdot\bar{0}=\bar{1}\cdot\bar{0}=\bar{0}$ and $\bar{1}\cdot\bar{1}=\bar{1}$: given $(A,g)$, the action of a $g\in G$ on $G/H = \{g_j H \mid j = 1, \dots, [G:H]\}$ gives a permutation matrix in $M_{[G:H]}(\C)$, i.e. column \(j\) represents the image of \(g_j H\) under \(g\). Now we can replace the $1$ in each column \(j\) by $\bar{0}$ if \(g_j H \subseteq A\), and by $\bar{1}$ if \(g_j H \nsubseteq A\). 
	
	Explicitly, this gives the following: \([(A, g)] \in S_H(G)\) is mapped to the matrix with entries 
	\[m_{ij} = \begin{cases}
	\bar{0} &\text{if } gg_j H = g_i H \text{ and } g_j H \subseteq A \\
	\bar{1} &\text{if } gg_j H = g_i H \text{ and } g_j H \nsubseteq A \\
	0 &\text{if } gg_j H \neq g_i H
	\end{cases}.\]
	Observe that each row and column has exactly one nonzero entry. One may check that this is indeed a morphism. 
\end{remark}

\section{An application: $\Sym_{n-1}\subset \Sym_n$}

In this section we apply our general theory to the important special case where $G$ is the symmetric group $\Sym_n$ and $H$ is the subgroup $\mathfrak{S}_{n}^{[1]}\cong \mathfrak{S}_{1}\times \mathfrak{S}_{n-1} \equiv \mathfrak{S}_{n-1}$ of the permutations fixing $1$.
This will provide a natural extension of the classical representation theory of $\Sym_n$.

\subsection{$\mathfrak{S}_{n-1}$-global $\mathfrak{S}_n$-partial representation theory} \label{sec:SnSnminus1}

Henceforth we use the notation introduced in Example~\ref{ex:SkSnk_action}. Moreover, we use freely classical definitions and results from the representation theory of $\Sym_n$ and its combinatorics: for these we refer to the standard \cite[Chapter~7]{Stanley-Book-1999}. 

\medskip

We start by looking at the algebra $\C\Gamma_{\Sym_{n-1}}(\Sym_{n})$. In order to apply Theorem~\ref{thm:algebra_decomposition} we need to understand the connected components of $\Gamma_{\Sym_{n-1}}(\Sym_{n})$ and the corresponding isotropy groups.

We already identified the action of $\Sym_n$ on the cosets $\Sym_n/\Sym_{n-1}$ with the defining action of $\Sym_n$ on the set $\binom{[n]}{1}=[n]$. So under this identification $\mathcal{P}_{\Sym_{n-1}}(\Sym_n/\Sym_{n-1})\equiv\{A\subseteq [n]\mid 1\in A \}$. Given $A\in \mathcal{P}_{\Sym_{n-1}}(\Sym_n/\Sym_{n-1})$ of cardinality $k\cdot |\Sym_{n-1}|=k\cdot(n-1)!$ with $k\geq 1$, it is clear that its stabilizer is $\Sym_n^{A}\cong \Sym_k\times \Sym_{n-k}$ (here we identify $A$ with the corresponding subset of $[n]$). Therefore $m_A=|A|/|\Sym_n^{A}|=\binom{n-1}{k-1}$, which is precisely the number of elements $A\in \mathcal{P}_{\Sym_{n-1}}(\Sym_n/\Sym_{n-1})$ of cardinality $k\cdot |\Sym_{n-1}|=k\cdot(n-1)!$. Hence for each $1\leq k\leq n$ there is precisely one connected component. So we can apply Theorem~\ref{thm:algebra_decomposition} to get the formula
\begin{equation}
\C\Gamma_{\Sym_{n-1}}(\Sym_{n})\cong \prod_{k=1}^nM_{\binom{n-1}{k-1}}(\C[\Sym_k\times \Sym_{n-k}]).
\end{equation}
Notice that this formula is in agreement with formula (\ref{eq:cardinality}) of the dimension of $\mathbb{C}\Gamma_{\mathfrak{S}_{n-1}}(\mathfrak{S}_n)$: indeed
\begin{align*} 
\sum_{k=1}^n\binom{n-1}{k-1}^2k!(n-k)! & =(n-1)!\sum_{k=1}^nk\binom{n-1}{k-1}\\
& =(n-1)!(n+1)2^{n-2}\\
& =2^{n-2}(n!+(n-1)!)\\
& =2^{|\mathfrak{S}_n/\mathfrak{S}_{n-1}|-2}(|\mathfrak{S}_n|+|\mathfrak{S}_{n-1}|).
\end{align*}

\begin{example}\label{ex:graphics2}
	Recall that $\Sym_3^{[1]} \equiv \Sym_2 = \{\mathrm{id},(2,3)\}$, $\Sym_3/\Sym_2 = \{\Sym_2,(1,2)\Sym_2,(1,3)\Sym_2\} \leftrightarrow \{1,2,3\} = [3]$ and that $\sigma\cdot(1,n)\Sym_2 = (1,m)\Sym_2$ if and only if $\sigma(n)=m$.
	In order to help intuition and visualization, the following picture represents the groupoid $\Gamma_{\Sym_2}(\Sym_3)$. 
	
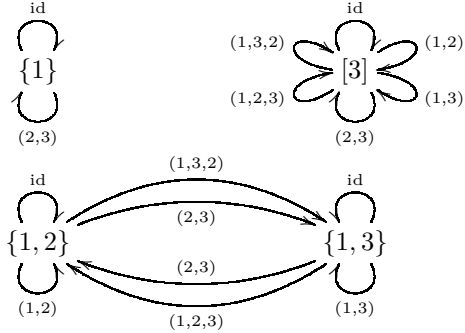
\[
	\begin{gathered}
	\xymatrix @!0 @C=60pt @R=65pt {
		{\{1\}} \ar @(ul,ur) []^-{\mathrm{id}} \ar @(dr,dl) []^-{(2,3)} & & {[3]} \ar @(ul,ur)[]^-{\mathrm{id}} \ar @(ur,r)[]^-{(1,2)} \ar @(r,dr)[]^-{(1,3)} \ar @(dr,dl)[]^-{(2,3)} \ar @(dl,l)[]^-{(1,2,3)} \ar @(l,ul)[]^-{(1,3,2)} \\
		{\left\{1,2\right\}} \ar @(ul,ur)[]^-{\mathrm{id}} \ar @(dl,dr)[]_-{(1,2)} \ar @/^1.3pc/[rr]_-{(2,3)}  \ar @/^2pc/[rr]^-{(1,3,2)} & & {\left\{1,3\right\}} \ar @(dr,dl)[]^-{(1,3)} \ar @(ur,ul)[]_-{\mathrm{id}} \ar @/^1.3pc/[ll]_-{(2,3)} \ar @/^2pc/[ll]^-{(1,2,3)} 
	}
	\end{gathered}
	\]
	It makes also evident that $\C\Gamma_{\Sym_2}(\Sym_3)\cong \C [\Sym_2] \times \C [\Sym_3] \times M_2(\C[\Sym_2\times \Sym_1])$, where $\Sym_2\times \Sym_1=\Sym_3^{\{3\}}=\{\mathrm{id},(1,2)\}$.
\end{example}

Now we want to apply Theorem~\ref{thm:irredsHglGpar} to construct all the irreducible $\Sym_{n-1}$-global $\Sym_n$-partial representations. 

Recall (cf.\ \cite[Chapter~7]{Stanley-Book-1999}) that the irreducible representations of $\Sym_{m}$ are indexed by the partitions of $m$: for any $m\geq 0$, let $\{V_{\mu}\}_{\mu\vdash m}$ be a complete set of inequivalent irreducible representations of $\Sym_{m}$ (where $\mu\vdash m$ means ``$\mu$ partition of $m$''). Then, given $[k]\subseteq [n]\equiv \Sym_{n}/\Sym_{n-1}$ with $k\geq 1$, a complete set of inequivalent irreducible representations of $\Sym_k\times \Sym_{n-k}$ is $\{V_\lambda\otimes V_\mu \}_{\lambda\vdash k, \mu\vdash n-k}$, so that a complete set of inequivalent irreducible $\Sym_{n-1}$-global $\Sym_n$-partial representations is given by
\begin{equation}
V_{(\lambda,\mu)}\coloneqq \Ind_{[k]}(V_\lambda\otimes V_\mu)\quad \text{ for }1\leq k\leq n,\lambda\vdash k, \mu\vdash n-k.
\end{equation}
In particular we have the formula for the dimension
\begin{equation}
\dim_\C V_{(\lambda,\mu)} = \binom{n-1}{k-1}f^\lambda f^{\mu},
\end{equation}
where for any partition $\nu$ we denote by $f^{\nu}$ the number of standard Young tableaux of shape $\nu$: this is a classical formula for the dimension of $V_\nu$ (cf.\ \cite[Chapter~7]{Stanley-Book-1999}).

\begin{remark}
	Specializing Remark~\ref{rem:S_H(G)_rep} to this situation, we get a representation of $S_{\Sym_{n-1}}(\Sym_n)$ in $M_n((\{\bar{0},\bar{1}\},\cdot))$, which is reminiscent of the hyperoctahedral group $\Sym_2 \wr \Sym_n$ viewed as the group of signed permutations. This is the only clue that we see it could have hinted to the striking similarity between the theory of $\Sym_{n-1}$-global $\Sym_n$-partial representations and the well-known representation theory of $\Sym_2 \wr \Sym_n$.
\end{remark}

\medskip

Now to study $\Res_{\Sym_{n-1}}^{\Sym_n}V_{(\lambda,\mu)}$, we need to better understand the construction of $\Ind_{[k]}(V_\lambda\otimes V_\mu)$.

Given $[k]\subseteq [n]\equiv \Sym_{n}/\Sym_{n-1}$ with $k\geq 1$, the corresponding subset of $\Sym_n$ is $A_k\coloneqq \{\sigma\in \Sym_n\mid \sigma(1)\in [k] \}$ so that $A_k^{-1}=\{\sigma\in \Sym_n\mid 1\in \sigma([k]) \}$. Since for every $\sigma\in A_k^{-1}$ we have $\sigma = \tau(1,\sigma^{-1}(1))$, where $\tau = \sigma (1,\sigma^{-1}(1)) \in \Sym_n^{[1]}$ and $(1,\sigma^{-1}(1))\in \Sym_n^{[k]}$, it is easy to see that \[ A_k^{-1}=\Sym_n^{[1]}\Sym_n^{[k]}=\Sym_{n-1} (\Sym_k\times \Sym_{n-k})=\Sym_{n-1} 1_{\Sym_n}(\Sym_k\times \Sym_{n-k}),\]
so that $A_k=(\Sym_k\times \Sym_{n-k})\Sym_{n-1}$.
Now, $(\Sym_k\times \Sym_{n-k})\cap \Sym_{n-1}=\Sym_n^{[k]}\cap \Sym_n^{[1]}\cong \Sym_1\times\Sym_{k-1}\times \Sym_{n-k}\equiv \Sym_{k-1}\times \Sym_{n-k}$, so we deduce from Example~\ref{ex:Hrestriction} that
\[ \mathsf{Res}_{\mathfrak{S}_{n-1}}^{\mathfrak{S}_n}V_{(\lambda,\mu)}=\Res_{\Sym_{n-1}}^{\Sym_n}(\Ind_{(\Sym_k\times \Sym_{n-k})\Sym_{n-1}}(V_\lambda\otimes V_\mu))\cong \Ind_{ \Sym_{k-1}\times \Sym_{n-k}}^{\Sym_{n-1}}(\Res_{\Sym_{k-1}\times \Sym_{n-k}}^{\Sym_{k}\times \Sym_{n-k}}(V_\lambda\otimes V_\mu)).  \]
Now we can use the classical formulas for the restriction and the induction of irreducibles of $\Sym_n$ to deduce the decomposition into irreducibles
\begin{equation} \label{eq:Sn-1restriction}
\mathsf{Res}_{\mathfrak{S}_{n-1}}^{\mathfrak{S}_n}V_{(\lambda,\mu)}\cong \bigoplus_{\lambda^1\to\lambda} \mathsf{Ind}_{\mathfrak{S}_{k-1}\times \mathfrak{S}_{n-k}}^{\mathfrak{S}_{n-1}} \left(V_{\lambda^1}\otimes V_{\mu}\right) \cong \bigoplus_{\nu\vdash n-1}  V_{\nu}^{\oplus d_{\lambda \mu}^\nu}
\end{equation}
with 
\begin{equation} \label{eq:dlambdamunu}
d_{\lambda \mu}^\nu\coloneqq \sum_{\lambda^1\to\lambda}c_{\lambda^1 \mu}^\nu, 
\end{equation}
where $c_{\lambda^1 \mu}^\nu$ are the \emph{Littlewood-Richardson coefficients}, and $\nu^1\to\nu$ indicates that $\nu$ covers $\nu^1$ in the Young lattice (i.e.\ $\nu^1$ is obtained by removing a corner from $\nu$): see \cite[Chapter~7]{Stanley-Book-1999} for the missing definitions and results.
\begin{remark}
	Formula \eqref{eq:Sn-1restriction} reduces to the corresponding Pieri rule when $V_{(\lambda,\mu)}$ is $\mathfrak{S}_n$-global, i.e.\ when $k = n$ whence $\mu=\varnothing$.
\end{remark}

\medskip

Also, by Theorem~\ref{thm:irredglobalization}, the globalization of $V_{(\lambda,\mu)}$ is given by \begin{equation}
\Ind_{ \mathfrak{S}_{k}\times \mathfrak{S}_{n-k}}^{\Sym_n}(V_\lambda\otimes V_\mu)\cong \bigoplus_{\rho\vdash n}V_\rho^{\oplus c_{\lambda \mu}^\rho},
\end{equation}
where we used again the same classical formulas.

\medskip

Finally, we want to apply the Frobenius reciprocity (\ref{s11eq:Frobenius}) to get formulas of the partial induction of an irreducible representation of $\Sym_{n-1}$ to $\Sym_n$. Given $\nu\vdash n-1$ and the corresponding irreducible $V_\nu$ of $\Sym_{n-1}$, we want to find a formula for $\PInd_{\Sym_{n-1}}^{\Sym_n}V_\nu$. Given $\lambda\vdash k$ and $\mu\vdash n-k$ with $1\leq k\leq n$, by Frobenius reciprocity we have
\[ \Hom_{\Sym_n}(\PInd_{\Sym_{n-1}}^{\Sym_n}V_\nu,V_{(\lambda,\mu)})\cong \Hom_{\Sym_{n-1}}(V_\nu,\Res_{\Sym_{n-1}}^{\Sym_n}V_{(\lambda,\mu)}).  \]
Combining this with \eqref{eq:Sn-1restriction} we get immediately the formula
\begin{equation} 
\mathsf{PInd}_{\mathfrak{S}_{n-1}}^{\mathfrak{S}_n} V_\nu = \bigoplus_{\lambda,\mu:|\lambda|+|\mu|=n} V_{(\lambda,\mu)}^{\oplus d_{\lambda \mu}^\nu},
\end{equation}
where $d_{\lambda \mu}^\nu$ is defined in \eqref{eq:dlambdamunu}.

In the next section we describe a situation that does not occur in general for any $G$ and $H$, but that is typical of towers of groups, like it is the case for the symmetric groups.

\subsection{Branching rules} \label{sec:branching}

Consider one of the irreducible $\mathfrak{S}_{n-1}$-global $\mathfrak{S}_n$-partial representations $V_{(\lambda,\mu)}$, where $\lambda\vdash k$ and $\mu\vdash n-k$ and where $k\geq 1$. Consider the subgroup $\mathfrak{S}_{n-1}':=\mathfrak{S}_n^{\{n\}}\equiv \mathfrak{S}_{n-1}\times \mathfrak{S}_1\subset \mathfrak{S}_n$. We want to describe the decomposition of the $\mathfrak{S}_{n-1}'$-partial representation $\mathsf{Res}_{\mathfrak{S}_{n-1}'}^{\mathfrak{S}_n}V_{(\lambda,\mu)}$ as a sum of irreducibles. 

To lighten the notation, we let $G\coloneqq \mathfrak{S}_{n}$, $H\coloneqq \mathfrak{S}_{n-1}$, $G'\coloneqq \mathfrak{S}_{n-1}'$, $H'\coloneqq G'\cap H=\mathfrak{S}_{n-2}'$, $K\coloneqq \mathfrak{S}_{k}\times \mathfrak{S}_{n-k}\subseteq G$ and $K'\coloneqq G'\cap K \equiv\mathfrak{S}_{k}\times \mathfrak{S}_{n-k-1}$.

The case $k=n$ corresponds to $\mathfrak{S}_n$-global representations, so it gives the well-known branching rule
\[ \mathsf{Res}_{\mathfrak{S}_{n-1}'}^{\mathfrak{S}_n}V_{(\lambda,\varnothing)}\cong \bigoplus_{\lambda^1\to\lambda}V_{(\lambda^1,\varnothing)}. \]

The case $k=1$ corresponds to the $\Sym_{n-1}$-global $\mathfrak{S}_n$-partial representations where any $\sigma\in \mathfrak{S}_n\setminus \Sym_{n-1}$ acts as $0$ (see Remark \ref{rem:lev1trivrep}).

It is easy to see that
\[ \mathsf{Res}_{\mathfrak{S}_{n-1}'}^{\mathfrak{S}_n}V_{((1),\mu)} = \mathsf{Res}_{\mathfrak{S}_{n-1}'}^{\mathfrak{S}_n}(\Ind_{[1]}(V_{(1)}\otimes V_\mu))\cong \bigoplus_{\mu^1\to \mu} \Ind_{[1]}(V_{(1)}\otimes V_{\mu^1})= \bigoplus_{\mu^1\to \mu} V_{((1),\mu^1)}.\]

From now on we assume $n>k> 1$.
By setting $\mathfrak{S}_{n-2}' \coloneqq \mathfrak{S}_{n-1}'\cap \mathfrak{S}_{n-1}=\mathfrak{S}_{n}^{\{1\}}\cap \mathfrak{S}_{n}^{\{n\}}$, it is clear that these will be actually irreducible $\mathfrak{S}_{n-2}'$-global $\mathfrak{S}_{n-1}'$-partial representations. 

We already identified in Example~\ref{ex:SkSnk_action} the action of $G$ on $X\coloneqq G/K$ by left multiplication as the action on $\binom{[n]}{k}$. Under this identification, there are two $G'$-orbits in $X$: the set $X'\coloneqq G'K/K$ of $k$-sets not containing $n$ and the set $\widetilde{X}\coloneqq G'(2,n)K/K$ of $k$-sets containing $n$.

In our identification (i.e.\ $\sigma\Sym_n^{[k]} \leftrightarrow \sigma([k])$), the set $Y\coloneqq HK/K\subset X=\binom{[n]}{k}$ corresponds to 
\[
Y= \left\{\sigma\Sym_n^{[k]}\mid \sigma \in \Sym_n^{[1]}\right\}= \left\{\left.A\in \binom{[n]}{k} ~\right|~ 1\in A \right\}
\] 
so that 
\[Y\cap X'=H'K/K=\left\{\left.A\in \binom{[n]}{k} ~\right|~ 1\in A,n\notin A  \right\}\] and \[Y\cap \widetilde{X}=H'(2,n)K/K=\left\{\left.A\in \binom{[n]}{k} ~\right|~ 1\in A,n\in A \right\}.\]

So if $(W,\rho)$ is an irreducible representation of $K$, then the $H$-global $G$-partial representation
\[\Ind_{KH}W\cong \bigoplus_{gK \in HK/K}W^{gK}\] decomposes, as $G'$-partial representation, as the direct sum of two subrepresentations
\begin{equation}\label{eq:twosubmodules}  
\bigoplus_{gK \in H'K/K}W^{gK}\oplus \bigoplus_{gK \in H'(2,n)K/K}W^{gK}.
\end{equation}
We begin by considering the restriction to $G'$ of the representation on the left in \eqref{eq:twosubmodules}. The assignment
\[ \varphi\colon H'K/K \to H'K'/K', \qquad gK \mapsto gK',  \]
is a well-defined isomorphism of $G'$-partial actions, so that we have an isomorphism of $G'$-partial representations
\[\bigoplus_{gK \in H'K/K}W^{gK}\cong \bigoplus_{gK' \in H'K'/K'}W^{gK'}=\Ind_{K'H'}(\Res_{K'}^K(W)). \]
Applying this to $W=V_\lambda\otimes V_\mu$ gives
\[ \bigoplus_{gK \in H'K/K}W^{gK}\cong \bigoplus_{\mu^1\to\mu}\Ind_{K'H'}(V_\lambda\otimes V_{\mu^1})= \bigoplus_{\mu^1\to\mu} V_{(\lambda,\mu^1)}. \]

Concerning the representation on the right in \eqref{eq:twosubmodules}, the argument is similar: the map
\[ \varphi'\colon H'(2,n)K/K \to H' K_{(2,n)}' /K_{(2,n)}', \qquad g(2,n)K \mapsto gK_{(2,n)}', \]
(where we recall that $K_{(2,n)}'\coloneqq (2,n) K' (2,n)$ as in Notation~\ref{notation}) is an isomorphism of $G'$-partial actions, so that we have a $G'$-isomorphism
\[\bigoplus_{g(2,n)K \in H'(2,n)K/K}W^{g(2,n)K}\cong \bigoplus_{gK' \in H'K_{(2,n)}'/K_{(2,n)}'}W_{(2,n)}^{gK_{(2,n)}'}=\Ind_{K_{(2,n)}'H'}
(\Res_{K_{(2,n)}'}^{K_{(2,n)}}W_{(2,n)}),
\]
where $W_{(2,n)}$ denotes the representation $\rho^{(2,n)}\colon K_{(2,n)} \to \GL(W),\, x \mapsto \rho\big((2,n)x(2,n)\big)$, again as in Notation~\ref{notation}. Applying this to $W=V_\lambda\otimes V_\mu$ gives
\[ \Res_{K_{(2,n)}'}^{K_{(2,n)}}W_{(2,n)}\cong \Res_{\Sym_{k-1}\times \Sym_{n-k}}^{\Sym_{k}\times \Sym_{n-k}} (V_\lambda\otimes V_\mu) \cong\bigoplus_{\lambda^1\to \lambda} V_{\lambda^1}\otimes V_\mu\]
so that
\[ \bigoplus_{g(2,n)K \in H'(2,n)K/K}W^{g(2,n)K}\cong \bigoplus_{\lambda^1\to \lambda} \Ind_{K_{(2,n)}'H'}(V_{\lambda^1}\otimes V_\mu) = \bigoplus_{\lambda^1\to \lambda}V_{(\lambda^1,\mu)}. \]
Finally we get the branching rule
\begin{equation} \label{eq:branchingrule}
\mathsf{Res}_{\mathfrak{S}_{n-1}'}^{\mathfrak{S}_n}V_{\lambda,\mu} \cong \bigoplus_{\lambda^1\to \lambda}V_{(\lambda^1,\mu)} \oplus \bigoplus_{\mu^1\to\mu} V_{(\lambda,\mu^1)} . 
\end{equation}

\section{Comments and future directions} \label{sec:comments}

One thing that we did not discuss here is the character theory. In fact there are general character theories (over $\C$), one for finite groupoids \cite{Ibort_Rodriguez-Book}, which can be applied to $\Gamma_H(G)$, and one for inverse semigroups \cite{Steinberg-Moebius1,Steinberg-Moebius2}, which can be applied to $S_H(G)$. Because of the isomorphism of algebras $\C\Gamma_H(G)\cong \C S_H(G)$, of course they will give the same results. In particular, the characters of $\C\Gamma_H(G)\cong \C S_H(G)$ can be recovered from the characters of the groups $K_{i,j}$ occurring in Theorem~\ref{thm:algebra_decomposition}.

\medskip

About future directions, there are several questions that arise naturally from this work. 

First of all, one can look at examples other than $G=\Sym_n$ and $H=\Sym_{n-1}\equiv \Sym_1\times \Sym_{n-1}$. It should be noted that already the cases $G=\Sym_n$ and $H=\Sym_{n-2}\equiv\Sym_1\times\Sym_1\times\Sym_{n-2}$ or $G=\Sym_n$ and $H=\Sym_{2}\times \Sym_{n-2}$ seem to be too complicated to compute explicitly. On the other hand, for example, it would be interesting to look into other Coxeter groups.

More generally, it would be interesting to see what properties of $G$ are determined by the $H$-global $G$-partial representations when $H$ is a characteristic subgroup of $G$, e.g.\ the derived group or the center of $G$. For example, it follows from \cite[Theorem 4.4]{Dokuchaev-Exel-Piccione} that the isomorphism class of $G/G'$ is determined by the partial representations of $G/G'$, where $G'$ is the derived subgroup of $G$. What can we say about $G$ by knowing the $G'$-global $G$-partial representations?

At a more speculative level, it would be interesting to see how the $H$-global $G$-partial representations are related to the Hecke algebra $\End_G (\C[G/H])\equiv \C[G]^{H\times H}$, especially in the case when $(G,H)$ is a Gelfand pair.

\medskip

Also, as the definition of an $H$-global $G$-partial representation makes sense also in infinite contexts, it would be interesting to look into infinite situations, like for example infinite compact groups, or Lie groups.

Even more generally, one could look to the case Hopf algebras, where similar notions can be defined, and they would be interesting to study. For example fixing a Hopf subalgebra $H$ of a given Hopf algebra, this in same cases might produce computable and interesting $H$-global partial representations.

\appendix
\section{Irreducibles of semisimple algebras}
\label{appendix:A}

In this appendix we outline how the representation theory of a finite-dimensional associative semisimple unital algebra $A$ gets recovered from the representation theory of the algebras $eAe$ for the idempotents $e\in A$. Along the way, we sketch a proof of Theorem~\ref{thm:reps_eAe}.

\medskip 

Let $A$ be a finite-dimensional associative semisimple unital algebra over $\mathbb{C}$, with $A\not\cong \mathbb{C}$. So, $A$ is the direct sum of matrix algebras by Wedderburn theory. Let $e\in A$ be a nontrivial idempotent of $A$, i.e.\ $0\neq e\neq 1$ (such an $e$ does exist since $A\not\cong \mathbb{C}$).

Given an $eAe$-module $W$ we define the $A$-module $\mathsf{Ind}_eW$ by setting
\[ \mathsf{Ind}_eW\coloneqq Ae\otimes_{eAe}W.  \]
Viceversa, given an $A$-module $V$, we define the $eAe$-module $\mathsf{Res}_eV$ by setting
\[ \mathsf{Res}_eV\coloneqq eV.  \]

Observe that for any $eAe$-module $W$ we have natural isomorphisms
\[  \mathsf{Res}_e( \mathsf{Ind}_eW )=e(Ae\otimes_{eAe}W)\cong eAe\otimes_{eAe}W \cong W. \]
Moreover
\[ Ae\, \mathsf{Ind}_eW= Ae(Ae\otimes_{eAe}W)=Ae(eAe)\otimes_{eAe}W =Ae\otimes_{eAe}W =\mathsf{Ind}_eW.\]
\begin{proposition}
	If $V$ is an $A$-module such that $\mathsf{Res}_eV=eV$ is an irreducible $eAe$-module and $AeV=V$, then $V$ is irreducible.
\end{proposition}
\begin{proof}
	If $V=V_1\oplus V_2$ as $A$-modules, then $eV=eV_1\oplus eV_2$; but $eV$ is irreducible, so without loss of generality $eV_2=0$. Now
	\[V=AeV=AeV_1\oplus AeV_2=AeV_1\subseteq V_1, \] which implies $V=V_1$. So $V$ is irreducible.
\end{proof}
Putting together the previous observation we get the following corollary, which is part of Theorem~\ref{thm:reps_eAe}.
\begin{corollary}\label{cor:A2}
	If $W$ is an irreducible $eAe$-module, then $\mathsf{Ind}_eW$ is irreducible.
\end{corollary}
\begin{proposition}\label{prop:A3}
	If $V$ is an irreducible $A$-module and $eV\neq 0$, then $eV$ is an irreducible $eAe$-module.
\end{proposition}
\begin{proof} 
	For any $0\neq v\in eV$ we have
	\[eAev=eAv=eV\]
	as $Av=V$ since $V$ is irreducible. 
\end{proof}
From the previous results we can easily deduce the following theorem, which is the remaining part of Theorem~\ref{thm:reps_eAe}.
\begin{theorem}
	Every irreducible $A$-module $V$ is isomorphic to $\mathsf{Ind}_eW$ for some nontrivial idempotent $e\in A$ and some irreducible $eAe$-module $W$.
\end{theorem}
\begin{proof}
	Let $V$ be an irreducible $A$-module. Since $A$ is semisimple, there exists a nontrivial idempotent $e\in A$ such that $eV\neq 0$ (otherwise $1$, which is a sum of nontrivial idempotents, acts as $0$). Then $eV$ is an irreducible $eAe$-module by Proposition \ref{prop:A3} and $\mathsf{Ind}_e(eV)$ is an irreducible $A$-module by Corollary \ref{cor:A2}. So we have a map of $A$-modules
	\[ \mathsf{Ind}_e(eV)=Ae\otimes_{eAe} eV\to V,\quad ae\otimes ev\mapsto aev,  \]
	which is clearly nonzero and hence an isomorphism by Schur's lemma.
\end{proof}

\bibliographystyle{amsalpha}
\bibliography{Biblebib}

\providecommand{\bysame}{\leavevmode\hbox to3em{\hrulefill}\thinspace}
\providecommand{\MR}{\relax\ifhmode\unskip\space\fi MR }
\providecommand{\MRhref}[2]{%
  \href{http://www.ams.org/mathscinet-getitem?mr=#1}{#2}
}
\providecommand{\href}[2]{#2}
\begin{thebibliography}{{Mac}87}

\bibitem[Aba03]{Abadie-globalization-2003}
F.~Abadie, \emph{Enveloping actions and {T}akai duality for partial actions},
  J. Funct. Anal. \textbf{197} (2003), no.~1, 14--67. \MR{1957674}

\bibitem[Aba18]{Abadie_dilations}
Fernando Abadie, \emph{Dilations of interaction groups that extend actions of
  {O}re semigroups}, J. Aust. Math. Soc. \textbf{104} (2018), no.~2, 145--161.
  \MR{3772668}

\bibitem[ABV15]{ParRep}
Marcelo Muniz~S. {Alves}, Eliezer {Batista}, and Joost {Vercruysse},
  \emph{{Partial representations of Hopf algebras.}}, {J. Algebra} \textbf{426}
  (2015), 137--187 (English).

\bibitem[ABV19]{Alves_Batista_Vercruysse-dilations}
\bysame, \emph{{Dilations of partial representations of Hopf algebras.}}, {J.
  Lond. Math. Soc., II. Ser.} \textbf{100} (2019), no.~1, 273--300 (English).

\bibitem[Bat17]{Batista-Survey-2017}
E.~Batista, \emph{Partial actions: what they are and why we care}, Bull. Belg.
  Math. Soc. Simon Stevin \textbf{24} (2017), no.~1, 35--71. \MR{3625785}

\bibitem[DE05]{Dokuchaev-Exel-TAMS-2005}
M.~Dokuchaev and R.~Exel, \emph{Associativity of crossed products by partial
  actions, enveloping actions and partial representations}, Trans. Amer. Math.
  Soc. \textbf{357} (2005), no.~5, 1931--1952. \MR{2115083}

\bibitem[DEP00]{Dokuchaev-Exel-Piccione}
M.~Dokuchaev, R.~Exel, and P.~Piccione, \emph{Partial representations and
  partial group algebras}, J. Algebra \textbf{226} (2000), no.~1, 505--532.
  \MR{1749902}

\bibitem[Dok19]{Dokuchaev-Survey-2019}
M.~Dokuchaev, \emph{Recent developments around partial actions}, S\~{a}o Paulo
  J. Math. Sci. \textbf{13} (2019), no.~1, 195--247. \MR{3947402}

\bibitem[Exe94]{Exel-CircleActions-1994}
R.~Exel, \emph{Circle actions on {$C^*$}-algebras, partial automorphisms, and a
  generalized {P}imsner-{V}oiculescu exact sequence}, J. Funct. Anal.
  \textbf{122} (1994), no.~2, 361--401. \MR{1276163}

\bibitem[Exe98]{Exel-semigroup-PAMS}
Ruy Exel, \emph{Partial actions of groups and actions of inverse semigroups},
  Proc. Amer. Math. Soc. \textbf{126} (1998), no.~12, 3481--3494. \MR{1469405}

\bibitem[{Exe}17]{Exel-book-2017}
Ruy {Exel}, \emph{{Partial dynamical systems, Fell bundles and applications.}},
  vol. 224, Providence, RI: American Mathematical Society (AMS), 2017
  (English).

\bibitem[IR]{Ibort_Rodriguez-Book}
Alberto Ibort and Miguel Rodriguez, \emph{An introduction to groups, groupoids
  and their representations}, CRC Press.

\bibitem[KL04]{KellenLawson}
J.~{Kellendonk} and Mark~V. {Lawson}, \emph{{Partial actions of groups.}},
  {Int. J. Algebra Comput.} \textbf{14} (2004), no.~1, 87--114 (English).

\bibitem[{Law}98]{Lawson}
Mark {Lawson}, \emph{{Inverse semigroups. The theory of partial symmetries.}},
  Singapore: World Scientific, 1998 (English).

\bibitem[{Mac}87]{Mackenzie}
K.~{Mackenzie}, \emph{{Lie groupoids and Lie algebroids in differential
  geometry.}}, vol. 124, Cambridge University Press, Cambridge. London
  Mathematical Society, London, 1987 (English).

\bibitem[Ser77]{SerreBook}
Jean-Pierre Serre, \emph{Linear representations of finite groups},
  Springer-Verlag, New York-Heidelberg, 1977, Translated from the second French
  edition by Leonard L. Scott, Graduate Texts in Mathematics, Vol. 42.
  \MR{0450380}

\bibitem[Sta99]{Stanley-Book-1999}
Richard~P. Stanley, \emph{Enumerative combinatorics. {V}ol. 2}, Cambridge
  Studies in Advanced Mathematics, vol.~62, Cambridge University Press,
  Cambridge, 1999, With a foreword by Gian-Carlo Rota and appendix 1 by Sergey
  Fomin. \MR{1676282}

\bibitem[Ste06]{Steinberg-Moebius1}
Benjamin Steinberg, \emph{M\"{o}bius functions and semigroup representation
  theory}, J. Combin. Theory Ser. A \textbf{113} (2006), no.~5, 866--881.
  \MR{2231092}

\bibitem[Ste08]{Steinberg-Moebius2}
\bysame, \emph{M\"{o}bius functions and semigroup representation theory. {II}.
  {C}haracter formulas and multiplicities}, Adv. Math. \textbf{217} (2008),
  no.~4, 1521--1557. \MR{2382734}

\bibitem[Ste16]{Steinberg-Book}
\bysame, \emph{Representation theory of finite monoids}, Universitext,
  Springer, Cham, 2016. \MR{3525092}

\bibitem[SV20]{Saracco_Vercruysse-Globalization}
Paolo Saracco and Joost Vercruysse, \emph{Globalization for geometric partial
  comodules}, 2020.

\end{thebibliography}

\end{document}